\documentclass[12pt]{amsart}
\usepackage{graphicx} 
\usepackage{xcolor}
\usepackage{amsmath}
\usepackage{amsfonts}
\usepackage{amsthm}
\usepackage{microtype}
\usepackage{amssymb}
\usepackage[hidelinks]{hyperref}
\usepackage{enumitem}
\usepackage{setspace}   
\usepackage{etoolbox} 
\usepackage{secdot}
\usepackage{geometry} 

\newtheorem{main}{Theorem}
\newtheorem{theorem}{Theorem}[section]
\newtheorem{corollary}[theorem]{Corollary}
\newtheorem{fact}[theorem]{Fact}
\newtheorem*{theorem*}{Theorem}
\newtheorem{prop}[theorem]{Proposition}
\newtheorem*{examples}{Examples}
\newtheorem{example}[theorem]{Example}
\newtheorem{lemma}[theorem]{Lemma}
\theoremstyle{definition}
\newtheorem{definition}[theorem]{Definition}

\theoremstyle{remark}
\newtheorem{remark}[theorem]{Remark}

\DeclareMathSymbol{\mlq}{\mathord}{operators}{``}
\DeclareMathSymbol{\mrq}{\mathord}{operators}{`'}

\DeclareMathOperator{\tp}{tp}

\DeclareMathOperator{\Th}{Th}

\DeclareMathOperator{\dcl}{dcl}
\DeclareMathOperator{\acl}{acl}

\DeclareMathOperator{\res}{res}
\DeclareMathOperator{\Res}{Res}

\DeclareMathOperator{\ACVF}{ACVF}

\DeclareMathOperator{\acvfball}{ACVF^{bl}}

\DeclareMathOperator{\fin}{\subseteq_{\text{fin}}}

\def\ind{\mathrel{\raise0.2ex\hbox{\ooalign{\hidewidth$\vert$\hidewidth\cr\raise-0.9ex\hbox{$\smile$}}}}}

\title{Relative distality and incidence bounds in positive characteristic}
\author{Mira Tartarotti}

\begin{document}

{
\let\clearpage\relax
\maketitle
}
\begin{abstract}
    We prove Szemerédi-Trotter-type incidence bounds for algebraic relations in any field $F$ admitting a valuation with finite residue field.
    This includes $\mathbb F_q(t)$, $\mathbb F_q((t))$, and all finitely generated extensions of finite fields. 
    Bays and Martin \cite{bays_martin} show that any algebraic $K_{d,s}$-free relation $E\subseteq F^k\times F^\ell$ admits a power-saving in the classical Zarankiewicz bound of Kővári, Sós, and Turán.
    Answering a question of the authors, we show that this power-saving is uniform in $F$:
    There is some $\epsilon>0$ depending only on $d$ and the complexity of $E$ (and not on $F$) such that for $A\subseteq F^k$ and $B\subseteq F^\ell$, the number of edges $|E(A,B)|$ is at most $O(|A|^{1-\epsilon}|B|^{1-\frac{1-\epsilon}{d}}+|A|+|B|)$.
    For the case where $A$ is a set of points and $B$ is a set of lines in $F^2$ and $E$ is the incidence relation, we prove the explicit bound $O(|P|^{\frac{8}{9}}|L|^{\frac{5}{9}}+|P|+|L|)$.
    As the central ingredient of the proof, we introduce the model-theoretic notion of distality \emph{relative to a definable set}. 
    We begin developing this notion and prove that the theory of algebraically closed valued fields is distal relative to the residue field. 
    An application of a result of Chernikov, Peterzil, and Starchenko \cite{chernikov2023modeltheoreticelekesszabostableominimal}, based on work of Chernikov, Galvin, and Starchenko \cite{chernikov2020cuttinglemmazarankiewiczsproblem}, yields our incidence bounds.
\end{abstract}

\section{Introduction}

A binary relation $E\subseteq X\times Y$ is called \emph{$K_{d,s}$-free} for positive integers $d$ and $s$ if $E$ contains no set $A\times B$ where $|A|\geq d$ and $|B|\geq s$. 
An application of Jensen's inequality yields the following classical result, which bounds the size of $E$ when $E$ is $K_{d,s}$-free.

\begin{fact}[Kővári-Sós-Turán \cite{Kvri1954OnAP}]\label{fact: Kövari-Sos-Turan}
    Suppose $E\subseteq X\times Y$ is $K_{d,s}$-free and $|X|=m$, $|Y|=n$. Then
    \[|E|=O_{d,s}(mn^{1-1/d}+n).\]
\end{fact}

A celebrated result of Szemerédi and Trotter \cite{Szemerdi1983ExtremalPI} improves on this bound when $E$ is the incidence graph of points and lines in the real plane.  

\begin{fact}[Szemerédi-Trotter \cite{Szemerdi1983ExtremalPI}]\label{fact: sz tr}
    Let $P$ be a set of $m$ points and $L$ a set of $n$ lines, respectively in $\mathbb R^2$. Then 
    \[|\{(p,l)\in P\times L~:~p\in l\}|= O(m^{2/3}n^{2/3}+m+n).\]
\end{fact}

Since two distinct points are incident with exactly one line, $E$ is $K_{2,2}$-free. 
Thus, Fact \ref{fact: Kövari-Sos-Turan} gives the upper bound $O(mn^{1/2}+n)$, and Fact \ref{fact: sz tr} gives a power saving whenever $m^{2-\epsilon}\geq n$ for some $\epsilon>0$. 
Bounds of a similar nature were proved for the more general case where $E$ is an \emph{algebraic} relation -- that is, a Boolean combination of zero-sets of polynomials, in $\mathbb C$ by Elekes and Szabó \cite[Theorem 9]{elekes2012find}. 
In particular, such bounds hold when $E$ is an algebraic relation in any field of characteristic zero, by completeness of the theory of algebraically closed fields.
Fox et al. \cite{fox2015semialgebraicversionzarankiewiczsproblem} show similar bounds when $E$ is \emph{semi-algebraic} in $\mathbb R$, i.e. a set defined by polynomial equalities and inequalities. 
In both cases, the power saving depends only on $d$ and $s$ and on the \emph{description complexity} of $E$, that is, the arity and maximal degree of polynomials involved in the definition of $E$.
Bays and Martin \cite{bays_martin} prove Szemerédi-Trotter-type bounds for algebraic relations in certain fields of positive characteristic, namely those which admit a valuation with finite residue field.
The power saving obtained in that paper depends on the size of that residue field. 
Answering \cite[Question 6.3]{bays_martin}, we show the following result, which asserts that this dependence can be removed.
For concreteness, we state the result in terms of the function field $\mathbb F_q(t)$, but the result holds for any field admitting a valuation with residue field $\mathbb F_q$. 

\begin{main}[Special case of Theorem \ref{thm: incidence bounds valued fields}]\label{thm: intro 1} 
    For positive integers $k$ and $d$, there is some $\epsilon>0$ depending only on $k$ and $d$, such that for any algebraic relation $E\subseteq \mathbb F_q(t)^r\times  \mathbb F_q(t)^\ell$ of description complexity at most $k$ and any finite $A\subseteq \mathbb F_q(t)^r$ and $B\subseteq \mathbb F_q(t)^\ell$ with the property that $E(A,B)$ is $K_{d,s}$-free for some $s\in \mathbb N$, we have
    \[\big|E(A,B)\big|=O_{d,s,k,q}\big(|A|^{1-\epsilon}|B|^{1-\frac{1}{d}+\frac{\epsilon}{d}}+|A|+|B|\big).\]
\end{main}

In fact, we prove a slightly more general version of Theorem \ref{thm: intro 1}, where $E$ is defined by any quantifier-free formula with parameters in $F$ in the language of valued fields.

We note two ways in which this result is sharp: 
Firstly, the assumption that $F$ has a valuation with finite residue field cannot be removed (though it is likely possible to weaken it).
Indeed, if $F$ is the algebraic closure of $\mathbb F_p$ and $E$ is point-line incidence, we can set $P_m$ and $L_m$ to be the sets of all points and lines respectively in $\mathbb F_{p^m}\leq F$, and we see that there are $p^{3m}$-many incidences between $P_m$ and $L_m$, so the classical bound from Fact \ref{fact: Kövari-Sos-Turan} is attained. 
Secondly, the dependence on $q$ of the constant factor in Theorem \ref{thm: intro 1} cannot be removed, for the same reason.
We remark in Section \ref{sec: applications} that the dependence can be shown to be polynomial in $q$ by carefully chasing constants; as we do not provide details, we omit this in the theorem statement.

Theorem \ref{thm: intro 1} does not establish quantitative estimates for $\epsilon$, though this is in theory possible: The proof is constructive, up to eliminating quantifiers in the theory of algebraically closed valued fields, which can be done effectively by \cite{weispfenning2006quantifier}.
For the
specific case of point-line incidence, we give an ad-hoc account of the argument for Theorem \ref{thm: intro 1}, to deduce the following explicit bound.

\begin{main}[Special case of Theorem \ref{thm: point line explicit bound}]\label{thm: intro 2}
    Let $P$ be a finite set of points and $L$ a finite set of lines, in $\mathbb F_q(t)^2$. Then
    \[\big|\{(p,l)\in P\times L:\ p\in l\}\big|=O_q\big(|P|^{\frac{8}{9}}|L|^{\frac{5}{9}}+|P|+|L|\big).\]
\end{main}

As above, the theorem holds when $\mathbb F_q(t)$ is replaced by an arbitrary field admitting a valuation with residue field $\mathbb F_q$, so in particular any finitely generated extension of a finite field.
We remark again that Theorem \ref{thm: intro 2} does not hold over arbitrary fields of positive characteristic and that the dependence of the constant factor on $q$ cannot be removed, by the same example as above.

For point-line incidence, non-trivial bounds in finite fields were established by Bourgain, Katz, and Tao \cite{bourgain2006sumproductestimatefinitefields}, followed by a series of extensions and improvements \cite{vinh2007szemereditrottertypetheoremsumproduct, Grosu_2014, Helfgott_2010,JONES2016136}. A common feature of these results is that they assume the set of points and lines to be either very small or very large with respect to the characteristic. Moreover, they do not say anything about infinite fields of positive characteristic. 

One motivation for studying geometric incidence bounds comes from their close connection to sum-product phenomena: Given a finite set $A\subseteq \mathbb R$, either $A+A$ or $A\cdot A$ must be large\footnote{i.e. $\geq |A|^{1+\epsilon}$}. In other words, finite sets of reals cannot be close to being both additively and multiplicatively closed. The first proof of this, given by Elekes \cite{GyorgyElekes1997}, is a direct application of Fact \ref{fact: sz tr}. Using this idea, one could derive from Theorem \ref{thm: intro 2} a sum-product bound for function fields of finite fields, recovering a quantitatively weaker version of a result of Bloom and Jones \cite{bloom2013sumproducttheoremfunctionfields}. 
In fact, one should be able to derive more general results in the spirit of Elekes-Rónyai \cite{ELEKES20001} and of Elekes-Szabó \cite{elekes2012find} for function fields of finite characteristic, as is done in \cite[Theorem 7.1]{bays_martin}. \\

Our methods are directly inspired by those in \cite{bays_martin}. 
In particular, the key to proving Theorem \ref{thm: intro 1} is a result due to Chernikov, Peterzil and Starchenko \cite{chernikov2023modeltheoreticelekesszabostableominimal}, which establishes Szemerédi-Trotter-type bounds for relations defined in so-called \emph{distal} first-order structures. 
That result relies, in turn, on the `cutting lemma' for such relations, due to Chernikov, Galvin, and Starchenko \cite{chernikov2020cuttinglemmazarankiewiczsproblem}.

The class of distal structures was first studied by Belegradek, Stolboushkin, and Taitslin \cite{belegradek1999extended} under the name \emph{isolation property}.
Distal structures, which include the $p$-adics, Presburger arithmetic, and all o-minimal structures,  were then defined independently by Simon \cite{Simon_2013}, with the aim of isolating the class of purely unstable NIP structures. 
This classification of NIP structures is not exhaustive: A typical example of an NIP theory that is neither distal nor stable is the theory of algebraically closed valued fields, denoted $\ACVF$.
Models of $\ACVF$ contain a stable part in the form of their residue field and an o-minimal part in the form of the value group. 
A common slogan is that $\ACVF$ is dominated by these two parts. 
This leads to the intuitive idea that algebraically closed valued fields should be in some sense distal `up to their residue field'. \\
The second key to proving Theorem \ref{thm: intro 1} is to make this intuition concrete. 
We define in Section \ref{sec: relative compbility and distality} a notion of {distality} of a structure (or theory) \emph{relative to a definable set}.
We leave a further development of relative distality, including further examples, to future work.
Our definition is in analogy to the combinatorial characterization of distality due to Chernikov and Simon \cite{chernikov2012externallydefinablesetsdependent}, in terms of so-called \emph{distal cell decompositions}, as introduced in \cite{chernikov2020cuttinglemmazarankiewiczsproblem}.
We show in Section \ref{sec: ACVF} that $\ACVF$ is distal relative to the residue field, in the sense defined in Section \ref{sec: relative compbility and distality}. 
Using the fact that any finitely generated field extension of $\mathbb F_p$ admits a valuation with finite residue field, this turns out to be enough to apply the distal bounds of Chernikov, Peterzil and Starchenko, in order to derive the incidence bounds in Theorem \ref{thm: intro 1}.

\subsection{Overview}

We introduce in Section \ref{chapter: NIP and distality} the combinatorial characterization of distality, and state the distal incidence bounds from \cite{chernikov2023modeltheoreticelekesszabostableominimal}, in a way that enables us to apply them in the context of relative distality.
In Section \ref{sec: relative compbility and distality}, we introduce distality relative-to-$D$ as a local property of a given formula $\varphi(x,y)$ in a structure, via a notion of {uniform strong honest definition} relative-to-$D$, which we call \emph{$D$-still}. 
We prove various useful properties of relative distality, which are used later.
We observe that the generic colored order from \cite{Simon_2013} is distal relative to the color classes, and that an ordered vector space $V$ with a predicate for a dense proper subspace $P$ is distal relative to the quotient $V/P$.
For the sake of brevity, we leave the further study of examples and properties of relative distality -- which we believe to be of separate model-theoretic interest -- to upcoming work.
Crucially, we prove in Lemma \ref{lem: relative distality in one variable enough} that existence of $D$-stills propagates from unary partitioned formulas $\varphi(x;y)$ (i.e. where $|x|=1$) to all formulas.
We prove distality of $\ACVF$ relative to the residue field in Section \ref{sec: ACVF}, relying on Robinson's quantifier elimination \cite{robinson1956complete}.
We give a description of unary quantifier-free definable sets (Lemma \ref{lemma: swiss cheeses}), refining \cite[Lemma 5.2]{bays_martin}, which allows us to construct explicit relative USHDs for unary quantifier-free formulas.
Finally, in Section \ref{sec: applications}, we prove Theorem \ref{thm: intro 1} and Theorem \ref{thm: intro 2} as an application.
To prove the quantitative bound given by Theorem \ref{thm: intro 2}, we give a careful account of the more general proof, exploiting the fact that the application of Lemma \ref{lem: relative distality in one variable enough} is straightforward for the case of the formula $x_2=x_1y_1+y_2$ defining point-line incidence.
The only obstruction to making the general incidence bounds explicit in this way is to bound how the description complexity of a relation is affected by taking projections; in other words, how adding a quantifier and eliminating it in $\ACVF$ affects the complexity of the formula.

\section*{Acknowledgements}

We thank Martin Bays for suggesting this project, for many helpful conversations, and his comments on a previous version of this manuscript.
We thank Udi Hrushovski for several useful conversations on relative distality.

\section*{AI Statement}

GPT-5.5 was used to proofread a previous version of this manuscript in July 2026, where it pointed out several typos and minor errors.
In addition, the model pointed out an error in step 2 of the proof strategy (as described in the paragraph above Lemma \ref{lemma: step 1 of point line incidence bound}) for Theorem \ref{thm: point line explicit bound}. 
Specifically, the formula $\theta_4$ was not suitably defined, and an additional input variable of the residue field sort was needed. 
This led the author to spot a more substantial gap in the proof of  Theorem \ref{thm: ACVF distal relative to k}, which was remedied -- without AI input -- by switching to a three-sorted language of valued fields, rather than a one-sorted one.\\
In August 2026, the author prompted GPT-5.5 to perform an explicit quantifier-elimination for the theory of algebraically closed valued fields, in order to simplify the proof of Proposition \ref{proposition: DCD for point line}.
The model responded by reminding the author that the statement ``If $x$ has property $P$ then $x$ is not equal to $y$" is logically equivalent to the statement ``$y$ does not have property $P$", which trivially concluded the task.\\
There was no further AI input to this manuscript. In particular, the writing was exclusively produced by the author.

\section{Background}\label{chapter: NIP and distality}

\subsection{Notation and basic model theory}

We assume knowledge of some basic model-theoretic concepts, for which we refer the reader to \cite{marker2006model}.

Let $L$ denote a first-order language, which is possibly many-sorted, and
let $M$ be an $L$-structure.
We write $\Th(M)$ for the $L$-theory of $M$.
For a given (possibly incomplete) $L$-theory $T$, we write $M\models T$ if $T\subseteq \Th(M)$.
We refer to (multi-)variables as variables, and use letters $x,y,z,u,v,w$ etc. to denote variables of arbitrary finite arity.
We write both $xy$ and $(x,y)$ to denote the concatenation of $x$ and $y$.
When $x$ is a variable, we write $M^x$ to denote tuples in $M$ of the same arity as $x$.
Note that here, the \emph{arity} of $x$ specifies both the length of $x$ and the sorts of the unary sub-variables of $x$.

When $A$ is a (possibly infinite) tuple or a subset of $M$, we write $L(A)$ to denote the expansion of $L$ by a constant for each element of $A$.
An $L$-formula is a formula \emph{without parameters}.
We say formula to refer to an $L$-formula when the ambient language is clear.
Similarly, a {definable set} is a set defined by a formula without parameters.
We write $\varphi(x)$ to indicate that the free variables of $\varphi$ are contained in $x$.
To distinguish from the logical symbol $=$, we write $\equiv$ for equality of formulas.
When $\varphi(x)$ is an $L(M)$-formula, we write $\varphi^M$ to denote the interpretation of $\varphi$ in $M$ -- that is, the subset of $M^x$ defined by $\varphi$.
For $a\in M^x$, we write $M\models\varphi(a)$ if $a\in\varphi^M$.
When the ambient structure $M$ is clear from the context, we simply write $\models\varphi(a)$.
If $A$ is a subset of $M$, we write $\varphi^M(A)$ to denote the restriction $\varphi^M\cap A^x$.
A tuple $a$ from (one of the sorts of) $M$ is \emph{definable} over a set $C\subseteq M$ if there is an $L(C)$-formula $\chi$ with $\chi^M=\{a\}$.
We write $\dcl(C)$ for the set of all such points $a$, and call this the \emph{definable closure} of $C$ in $M$.

A (binary) \emph{partitioned formula} is a formula of the form $\varphi(x,y)$ where $x$ and $y$ are disjoint variables, which we view as two separate parts of the variable-input of $\varphi$.
We often use a semi-colon to emphasize the partition of the variables, as in $\varphi(x;y)$ or $\psi(x;y;z)$. 
Given a partitioned formula $\varphi(x;y)$, a tuple $a\in M^x$ and a set $B\subseteq M^y$, we write $\tp_\varphi(a/B)$ for the $\varphi$-\emph{type} of $a$ over $B$, that is the set
\[\{\varphi(x;b):b\in B,\ M\models\varphi(a,b)\}\cup\{\neg\varphi(x;b):b\in B,\ M\models\neg\varphi(a,b)\}.\]

When $A\subseteq M^x$, we write $\varphi^M(A,B)$ to denote the restriction of $\varphi^M$ to $A\times B$. 

When $E\subseteq X\times Y$ is a binary relation on sets $X$ and $Y$, and $A\subseteq X$ and $B\subseteq Y$, we write $E(A,B)$ for the restriction of $E$ to $A\times B$.
When $a\in X$ is a point, we write $E(a,B)$ for $E(\{a\},B)\subseteq\{a\}\times B$, which we identify with its projection to $B$.
We sometimes write $E_a$ or $E(a)$ for $E(a,Y)$ and call this the \textit{fiber} of $E$ above $a$.
By abuse of notation, we also write $E_b$ or $E(b)$ to denote $E(X,b)$ when $b\in Y$, and call this the fiber of $E$ above $b$.
The intended meaning will be clear from the context.

We write $A\fin X$ if $A$ is a \emph{finite} subset of $X$. 
For a cardinal $\kappa$, we write $A\subseteq_\kappa X$ if $A\subseteq X$ with $|A|=\kappa$. 
We use $[n]$ to denote $\{1,2,\dots,n\}$.

\subsection{Distal cell decompositions}

The notion of distality was defined originally in terms of indiscernibles in \cite{Simon_2013} and rephrased in terms of global types in \cite[Chapter 9]{NIP_guide}. Chernikov and Simon give in \cite[Theorem 21]{chernikov2012externallydefinablesetsdependent} a characterization of distality as a ``local''\footnote{Local in that it is characterized by a property of each formula separately, but non-local in that this property depends on the rest of the structure.} combinatorial property.
We present this characterization in terms of \emph{distal cell decompositions}, as introduced in \cite{chernikov2020cuttinglemmazarankiewiczsproblem}. 

\begin{definition}[\cite{chernikov2020cuttinglemmazarankiewiczsproblem}]\label{def: dcd}
Let $\mathcal{E}$ be a family of relations $E\subseteq X\times Y$.
A \emph{cell decomposition} for $\mathcal{E}$ is a map $\mathcal{T}$, which assigns to each finite set $B\subseteq Y$ a collection $\mathcal{T}(B)\subseteq \mathcal{P}(X)$ of \emph{cells}, such that for every $B\fin Y$ with $|B|\geq 2$,
\begin{itemize}
    \item for every cell $\Delta\in\mathcal{T}(B)$ and $b\in B$ and $E\in\mathcal{E}$, either $\Delta\subseteq E_b$ or $\Delta\cap E_b=\emptyset$;
    \item $X$ is contained in $\bigcup \mathcal{T}(B)$.
\end{itemize}
For $t\in\mathbb N$ and $C>0$, we say that $\mathcal{T}$ has \emph{exponent} $\leq t$ with constant factor $C>0$ if for every finite $B\subseteq Y$, we have
\[\big|\mathcal{T}(B)\big|\leq C|B|^t.\]

We say that $\mathcal{T}$ is a \emph{distal} cell decomposition (DCD) if there exists a graph $R\subseteq X\times Y^s$ and a map $I:Y^s\to \mathcal{P}(Y)$  such that for every finite set $B\subseteq Y$,
\[\mathcal{T}(B)=\{R_c: c\in B^s\wedge I(c)\cap B=\emptyset\}.\]
In this case, we say that $R$ \emph{induces} $\mathcal{T}$.
\end{definition}

For a single graph $E$, we say that $\mathcal{T}$ is a DCD for $E$ when it is a DCD for $\{E\}$, in the sense of Definition \ref{def: dcd}.

\begin{remark}\label{remark: coding piecewise DCDs}
    The assumption that $|B|\geq 2$ in Definition \ref{def: dcd} allows for the following ``coding trick''.
    Given a family $\mathcal{E}$ as in Definition \ref{def: dcd} and a finite family $(R_i)_{i\leq k}$ of graphs $R_i\subseteq X\times Y^{t_i}$, suppose there is a cell decomposition for $\mathcal{E}$ whose cells are all of the form $(R_i)_c$ for some $i$ and $c\in Y^{t_i}$.
    Then $\mathcal{E}$ has a DCD.
    Namely, for $t:=\max_i t_i$, define a graph
    $R\subseteq X\times Y^{k+t+1}$ where $(a,c,d_0,\dots,d_{k})\in R$ whenever there is some $i\leq k$ with $d_0=d_i\neq d_j$ for each $j\neq i$ and $(a,c)\in R_i$.
    Now $R$ induces a DCD for $\mathcal{E}$, of exponent $\leq t$ with constant factor $k$.
    We may call this the DCD induced by the family $(R_i)_{i\leq k}$.
\end{remark}

We can now give an account of distality, using a characterization in terms of distal cell decompositions, essentially from \cite{chernikov2012externallydefinablesetsdependent}.

\begin{definition}\label{def: USHD}
   ${M}$ is \emph{distal} if for every $L$-formula $\varphi(x;y)$ there is some $L$-formula $\psi(x;z)$ such that $z$ is a concatenation of disjoint copies of $y$ and $\psi^M\subseteq M^{x}\times (M^{y})^{{|z|}/{|y|}}$ induces a DCD for $\varphi^{M}\subseteq M^{x}\times M^{y}$.
\end{definition}

It is worth noting that the property that $\psi^M$ induces a DCD for $\varphi^M$ is elementary in $M$; it holds if and only if $M$ is a model of a countable collection of $L$-sentences
\[\forall y_1\dots y_n\forall x\bigvee_{z\in\{y_1,\dots,y_n\}^t}\Big(\psi(x,z)\wedge\forall x'\big(\psi(x',z)\rightarrow \bigwedge_{i=1}^n\varphi(x',y_i)\leftrightarrow\varphi(x,y_i)\big)\Big)\]
with $n$ ranging over $\mathbb N$. 
Thus, we call a complete theory $T$ distal whenever some (each) model of $T$ is distal.

\begin{examples}
    \begin{itemize}
        \item The theory $\Th(\mathbb N,=)$ of infinite sets is not distal. 
        More generally, any stable theory with infinite models is non-distal, as witnessed by the formula $\varphi(x;y):\equiv x=y$.
        \item  The theory $\Th(\mathbb Q,<)$ of dense linear orders is distal by quantifier elimination.
        \item Presburger arithmetic $\Th(\mathbb Z,+,0,<)$ is distal by quantifier elimination in the expanded language $(+,0,<,(\equiv_m)_{m\in\mathbb N})$, where $\equiv_m$ is equality modulo $m$.
        \item The theory $\Th(\mathbb R,+,\cdot,0,1,<)$ of real closed fields is distal. More generally, any o-minimal theory is distal.
    \end{itemize}
\end{examples}

Chernikov, Peterzil and Starchenko \cite{chernikov2023modeltheoreticelekesszabostableominimal} show that graphs with DCDs admit non-trivial Zarankiewicz bounds, in the following sense.

\begin{theorem}\cite[Theorem 2.8]{chernikov2023modeltheoreticelekesszabostableominimal}\label{thm: distal zarankiewicz}
    Suppose $E\subseteq X\times Y$ has a DCD of exponent $t$ with constant factor $C>0$.
    Then for any $A\fin X$ and $B\fin Y$ and $d,s\in\mathbb N$ such that $E(A,B)$ is $K_{d,s}$-free,  
    \begin{equation}\label{eq: distal zarankiewicz bound}
        |E(A,B)|= O_{d,s,C,t}\big( |A|^{1-\epsilon}|B|^{1-\frac{1-\epsilon}{d}}+|A|+|B|\big)
    \end{equation}
    where $\epsilon=\frac{d-1}{td-1}$.
\end{theorem}

The result relies on the \emph{distal cutting lemma} from \cite{chernikov2020cuttinglemmazarankiewiczsproblem}.

\begin{remark}\label{remark: polynomial constant}
    A careful reading of the proof of the distal cutting lemma \cite[Theorem 3.2]{chernikov2020cuttinglemmazarankiewiczsproblem}, as well as \cite[Theorem 2.8]{chernikov2023modeltheoreticelekesszabostableominimal} and the Zarankiewicz bound for NIP-graphs from \cite[Theorem 2.1]{fox2015semialgebraicversionzarankiewiczsproblem}, shows that the constant factor in (\ref{eq: distal zarankiewicz bound}) is polynomial in $C$ and linear in $s$.
\end{remark}

\section{Relative distality}\label{sec: relative compbility and distality}

We introduce in this section a notion of distality \emph{relative to a definable set} via distal cell decompositions.
Fix throughout a countable first-order language $L$ with a distinguished definable set $D$, and an $L$-structure $M$. 
Write $D(M)$ for the interpretation of $D$ in $M$.
Throughout, we write $u_D$, $v_D$ etc. for variables made up of finitely many disjoint copies of a variable of the same arity as $D$.
When $u_D$ consists of $k$-many such copies, we write $D(M)^{|u_D|}$ for $D(M)^k$.

\begin{definition}
    Let $a\in M^x$ and let $p\subseteq \tp(a/M)$ be a partial type.
    We say that a formula $\psi(x)\in \tp(a/M)$ \emph{isolates} $p$ if $\psi(M)\subseteq p(M)$.
\end{definition}
Note that an isolating formula may contain the parameter $a$ itself.
For example, the formula $x=a$ trivially isolates $\tp(a/M)$.
Below, we will consider isolating formulas which contain an image of $a$ under a definable map into $D$ as parameters.
In what follows, $t$ is a positive integer and $z$ is a concatenation of $t$-many disjoint copies of the variable $y$.

\begin{definition}\label{def: relative distality on a set}
    Let $\varphi(x;y)$ and $\psi(x;z;u_D)$ be partitioned $L(M)$-formulas. We say that
    \begin{enumerate}
        \item[(i)] $\psi$ is a \emph{$D$-still}\footnote{The intended meaning of the word \emph{still} is an apparatus used for distillation.} for $\varphi$ if there is a definable map $f(x;z)$ into $D(M)^{|u_D|}$ such that for every subset $B\fin M^{y}$ with $|B|\geq 2$ and every $a\in {M}^{x}$, there is some $c\in B^t$ such that $\psi(x;c;f(a,c))$ isolates $\tp_\varphi(a/B)$; 
        \item[(ii)] ${M}$ is \emph{distal relative to $D$} if every partitioned $L$-formula $\varphi(x;y)$ has a $D$-still in ${M}$.  
    \end{enumerate}
    We call $t=|z|/|y|$ the \emph{arity} of the $D$-still $\psi$ for $\varphi$, and we say that $f$ \emph{witnesses} that $\psi$ is a $D$-still for $\varphi$ when (i) holds with $f$.
\end{definition}

Note that it is sensible also to make these definitions at the level of a (possibly incomplete) theory $T$ rather than a model.

\begin{definition}\label{def: relative distality for a theory}
    Let $\varphi(x;y)$ and $\psi(x;z;u_D)$ be partitioned $L$-formulas. We say that 
    \begin{enumerate}
        \item[(i)] $\psi$ is a {$D$-still} for $\varphi$ \emph{in $T$} if it is a $D$-still for $\varphi$ in every model of $T$; 
        \item[(ii)] $T$ is \emph{distal relative to $D$} if every partitioned $L$-formula $\varphi(x,y)$ has a $D$-still in $T$. 
    \end{enumerate}
\end{definition}

When the model $M$ or the theory $T$ is clear from the context, we may refer to $\psi$ as a $D$-still without specifying $M$ or $T$.
Consider the following trivial examples.

\begin{example}
\begin{itemize}
    \item[(i)]
    If $M$ has finitely many sorts and each sort has a constant, then ${M}$ is distal relative to its whole universe $M$.
    \item[(ii)] 
    $M$ is distal relative to $\emptyset$ if and only if $M$ is distal.
\end{itemize}
\end{example}

\begin{remark}\label{remark: facts about relative USHDs}
\begin{enumerate}
    \item[(i)] An $\emptyset$-still is a \emph{uniform strong honest definition} in the sense of \cite{bays_martin}.
    \item[(ii)] The assumption in Definition \ref{def: relative distality on a set} (i) that $f$ and $\psi$ share an input variable $z$ can always be achieved by adding dummy variables.
    \item[(iii)] The assumption that $|B|\geq 2$ allows us to \emph{code} several formulas into one when constructing $D$-stills. 
    Let $z^i$ be a variable consisting of finitely many disjoint copies of $y$.
    Suppose $(\psi_i(x;z^i;u^i_D))_{i\leq n}$ is a family of formulas and $(f_i(x;z^i))_{i\leq n}$ is a family of definable maps, such that for any $a\in M^x$ and $B\fin M^y$, there are $i\leq n$ and a tuple $c$ from $B$ such that $\psi_i(x;c;f(a,c))$ isolates $\tp_\varphi(a/B)$. 
    Setting $z=z^1\dots z^n$, $u_D=u_D^1\dots u_D^n$ and $s=y_0\dots y_n$ to be a concatenation of disjoint copies of $y$, the formula 
    $\psi(x;z,s;u_D):\equiv \bigvee_{i\leq n}\big(y_0=y_i\wedge \psi_i(x;z^i;u^i_D)\big)$
    is a $D$-still for $\varphi$, witnessed by $f(x;z,s):=(f_i(x;z^i))_{i\leq n}$.
    The assumption $|B|\geq 2$ allows for the correct choice of $\psi_i$ to be encoded in the ``$s$-part" of $c\in B^{zs}$.
    \item[(iv)] $\varphi$ has a D-still if and only if there is a formula $\theta(x;z)$ and a definable map $f(x;z)$ such that for $a\in M^x$ and $B\fin M^y$, there is $c\in B^t$ such that $\theta(x;c)\wedge f(x;c)=f(a;c)$ isolates $\tp_\varphi(a/B)$. 
\end{enumerate}
\end{remark}

When $(\psi_i(x;z^i;u_D^i))_i$ and $(f_i(x;z^i))_i$ are as in Remark \ref{remark: facts about relative USHDs} (iii), we refer to $(\psi_i)_i$ as a \emph{piecewise $D$-still} for $\varphi(x,y)$, witnessed by $(f_i)_i$.
It will be useful at certain points to preserve this information, rather than coding $(\psi_i)_i$ in one formula.
Write $t_i:=\frac{|z^i|}{|y|}$ and call $\max_it_i$ the \emph{arity} of the piecewise $D$-still $(\psi_i)_i$.
The motivation for this, and more generally, for Definition \ref{def: relative distality on a set}, is the following observation.

\begin{prop}\label{prop: k-stills give DCDs}
    Let $\varphi(x;y)$ be a formula with a piecewise $D$-still $(\psi_i(x;z^i,u^i_D))_{i\leq k}$ of arity $t$, witnessed by $(f_i(x;z^i))_{i\leq k}$, in $M$.
    Let $X\subseteq M^x$ and $Y\subseteq M^y$ be such that $D_0:=\bigcup_if_i(X;Y^{t_i})$ is finite.
    Then $(\psi_i^M)_i$ induces a DCD of exponent $t$ with constant factor $k|D_0|$ for $\varphi^M(X,Y)\subseteq X\times Y$.
\end{prop}
\begin{proof}
    For any $i\leq k$ and $d\in f_i(X,Y^{t_i})$, define
    \[R_{i,d}:=\psi_i^M(X,Y;d)=\big\{(a,c)\in X\times Y^{t_i}:M\models \psi_i(a,c,d)\big\}.\]
    By assumption and Remark \ref{remark: coding piecewise DCDs}, we can code the family $(R_{i,d})_{i,d}$ into one graph, which induces a DCD of exponent $t$ with constant factor
    \[|\big\{(i,d):i\leq k,\ d\in D_0\big\}|\leq k|D_0|.\]
\end{proof}

Applying Theorem \ref{thm: distal zarankiewicz}, we immediately get the following incidence bound.

\begin{corollary}\label{cor: k-stills give zarankiewicz bounds}
    Let $\varphi(x;y)$ be a formula with a piecewise $D$-still of arity $t$ in $M$.
    Let $X\subseteq M^x$ and $Y\subseteq M^y$ be such that $D_0:=D(\dcl(XY))$ is finite, say $q:=|D_0|$.
    Then for any positive integer $d$, there is some $\epsilon>0$ depending only on $\Th(M)$ and $d$, such that for any positive integer $s$ and finite sets $A\subseteq X$ and $B\subseteq Y$, if $\varphi^M(A,B)$ is $K_{d,s}$-free, then  
     \[|\varphi^M(A,B)|=O_{d,s,t,q}\big(|A|^{1-\epsilon}|B|^{1-\frac{1-\epsilon}{d}}+|A|+|B|\big).\]
\end{corollary}
\begin{proof}
    This follows from an application of Theorem \ref{thm: distal zarankiewicz} to Proposition \ref{prop: k-stills give DCDs}.
\end{proof}

\begin{remark}\label{remark: poly dependence relative distality bound}
    By Remark \ref{remark: polynomial constant}, the constant factor in the above bound is polynomial in $q$.
\end{remark}

We show now that if $M$ is distal relative to $D$, then formulas with parameters also have $D$-stills, and these are uniform for definable families of formulas $(\varphi(x,y;b))_{b\in M}$.
Recall that whenever $\psi(x;z;u_D)$ is a $D$-still for a formula $\varphi(x;y)$, then in particular $z$ is a concatenation of disjoint copies of $y$.

\begin{lemma}\label{lem: relative USHDs with parameters}
    Suppose $\psi(x,{z},u_D)$ is a $D$-still for $\varphi(x;y,y')$, witnessed by $f(x,z)$.
    For each $b'\in M^{y'}$, let $\psi'_{b'}(x,z',u_D)$ denote the formula obtained by replacing each disjoint copy of $y'$ in $z$ by $b'$ in the formula $\psi$.
    Define $f'_{b'}(x,z')$ similarly.
    Then $\psi'_{b'}(x,z',u_D) $ is a $D$-still for the $L(b')$-formula $\varphi_{b'}(x,y):\equiv \varphi(x,y,b')$.
\end{lemma}
\begin{proof}
    This follows from the fact that the parameter $c$ appearing in Definition \ref{def: relative distality on a set} (i) is a concatenation of parameters from $B$.
\end{proof}

The following lemma reduces the task of proving relative distality of a structure to proving it for all \emph{unary} formulas, as is the case for NIP and distality.

\begin{lemma}\label{lem: relative distality in one variable enough}
    Suppose each $L$-formula $\varphi(x,y)$ with $|x|=1$ has a $D$-still in $M$.
    Then $M$ is distal relative to $D$.
\end{lemma}
\begin{proof}
    We essentially repeat the proof of \cite[Proposition 1.9]{aschenbrenner2022distalityvaluedfieldsrelated}. 
    The proof is by induction. 
    Let $\varphi(x,x';y)$ be an $L$-formula with $|x|=N$ and $|x'|=1$. 
    By the inductive hypothesis, the formula $\xi(x;y,x'):\equiv \varphi(x,x',y)$ has a $D$-still, so by Lemma \ref{lem: relative USHDs with parameters}, there is a formula $\psi(x,z,u_D,x')$ and a definable map $f(x,z,x')$, such that for any $a'\in M^{x'}$, the $L(a')$-formula $ \psi_{a'}(x,z,u_D):=\psi(x,z,u_D;a')$ is a $D$-still for $\xi_{a'}(x,y):\equiv \xi(x,y;a')$, witnessed by $f_{a'}(x,z):=f(x,z;a')$.
    Let
    \[\theta^\epsilon(x',y,z,u_D):=\forall x(\psi(x,z,u_D,x')\rightarrow \varphi^\epsilon(x,x',y))\]
    with $\epsilon\in \{+,-\}$ and $\varphi^+:=\varphi$, $\varphi^-:=\neg\varphi$.
    Applying the inductive hypothesis and Lemma \ref{lem: relative USHDs with parameters} to the formula $\theta^\epsilon(x';y,z,u_D):\equiv \theta^\epsilon(x',y,z,u_D)$, we find a formula $\chi^\epsilon(x',s^\epsilon,w^\epsilon_D,u_D)$ and a definable map $g^\epsilon(x',s^\epsilon,u_D)$ such that for any $d\in D^{|u_D|}$, the formula $\chi^\epsilon_d(x',s^\epsilon,w^\epsilon_D):\equiv \chi^\epsilon(x',s^\epsilon,w^\epsilon_D;d)$ is a $D$-still for $\theta^\epsilon_d(x';y,z):=\theta^\epsilon{(x',y,z;d)}$, witnessed by $g^\epsilon_d(x',s^\epsilon):=g^\epsilon(x',s^\epsilon;d)$, in $M(d)$.
    Then
    \[\rho(x,x';z,s^+,s^-;u_D,w^+_D,w^-_D):\equiv \psi(x,z,u_D,x')\wedge \chi^+(x',s^+,w^+_D,u_D)\wedge\chi^-(x',s^-,w^-_D,u_D)\]
    is a $D$-still for $\varphi(x,x';y)$, witnessed by 
    \[h(x,x';z,s^+,s^-):=\big(f(x,z,x'),g^+(x',s^+,f(x,z,x')),g^-(x',s^-,f(x,z,x'))\big).\]
    Indeed, let $a\in M^x$, $a'\in M^{x'}$ and $B\fin M^y$. 
    Find $c\in B^z$ such that $M\models \psi_{a'}(a,c,f(a,c,a'))$ and, setting $d:=f(a,c,a')$,
    \begin{align}\label{eq: transitivity of relative distality}
        M\models \forall x\big(\psi_{a'}(x,c,d)\rightarrow\bigwedge_{b\in B}\varphi(x,a',b)\leftrightarrow\varphi(a,a',b)\big).
    \end{align}
    Then, for $\epsilon\in\{+,-\}$, find $e^\epsilon\in B^{|s^\epsilon|/|y|}$ such that $M\models\chi^\epsilon_d(a',e^\epsilon,g^\epsilon(a',e^\epsilon,d))$ and
    \begin{align}\label{eq: second part of transitivity of relative distality}
        M\models \forall x'\Big(\chi^\epsilon_d(x',e^\epsilon,g^\epsilon(a',e^\epsilon,d))\rightarrow\bigwedge_{b\in B}\big(\theta^\epsilon_d(x',b,c)\leftrightarrow\theta^\epsilon_d(a',b,c) \big)\Big).
    \end{align}
    Set $d^\epsilon:=g^\epsilon(a',e^\epsilon,d)$. Then $M\models \rho(a,a';c,e^+,e^-;d,d^+,d^-)$.\\
    Suppose now $(\Tilde{a},\Tilde{a}')\models \rho(x,x';c,e^+,e^-;d,d^+,d^-)$ and let $b\in B^y$. 
    It suffices to show that $\models \varphi(a,a',b)\leftrightarrow \varphi(\Tilde{a},\Tilde{a}',b)$. Suppose  $\models \varphi(a,a',b)$.
    Then by (\ref{eq: transitivity of relative distality}), we have $\models\theta^+_d(a',b,c)$. 
    Since $ \models\chi^+_d(\Tilde{a}',e^+,d^+)$ by assumption,
    it follows from (\ref{eq: second part of transitivity of relative distality}) that also
    $\models\theta^+_d(\Tilde{a}';b,c)$.
    By definition of $\theta^+$ and the fact that $\models\psi_{\Tilde{a}'}(\Tilde{a},c,d)$, it follows that 
    $\models\varphi(\Tilde{a},\Tilde{a}',b)$.
\end{proof}

Another basic but key property of relative distality is that it is preserved under Boolean combinations of formulas.

\begin{lemma}\label{lem: Boolean combinations}
    A Boolean combination of formulas which have $D$-stills has a $D$-still.
\end{lemma}
\begin{proof}
    Any $D$-still for $\varphi(x,y)$ is also a $D$-still for $\neg\varphi(x,y)$. Suppose $\psi_1(x^1,z^1,u^1_D)$ is a $D$-still for $\varphi_1(x^1,y^1)$ and $\psi_2(x^2,z^2,u_D^2)$ is a $D$-still for $\varphi_2(x^2,y^2)$. Then the formula $\psi(x^1,x^2;z^1, z^2; u_D^1, u_D^2)$ defined as
    $\psi^1(x^1,z^1,u^1_D)\wedge\psi^2(x^2,z^2,u^2_D)$ is a $D$-still for $\varphi^1\wedge\varphi^2$.
\end{proof}

In some of the examples that follow, $D$ is an \emph{imaginary sort} of $M$. 
That is, $D$ is the quotient of a definable set $X$ by a definable equivalence relation $R$. 
In this case, $L$ has a natural expansion by a new sort for $D$ and a new definable map $\pi: X\to D$, and $M$ has a natural $L(D)$-expansion by setting $D(M):=X/R$ and $\pi(x):=x/R$.

\begin{example}\label{example: ordered equivalence}
    Let $M=(X,<,R)$ be the `colored order' from \cite{Simon_2013}, that is $(X,<)$ is a dense linear order and $R$ is an equivalence relation with infinitely many classes, which is generic with respect to the order $<$ in the sense that every $R$-class is order-dense. 
    Using quantifier elimination, one can see that $M$ is NIP. 
    On the other hand, $M$ is clearly unstable, and it is non-distal, as witnessed by the formula $R(x,y)$. 
    However, $M$ is distal relative to the imaginary sort $D:=M/R$.
    By quantifier elimination, Lemma \ref{lem: relative distality in one variable enough} and Lemma \ref{lem: Boolean combinations}, it is enough to consider unary atomic formulas. The formula $x<y$ is distal (i.e. has an $\emptyset$-still) by distality of dense linear orders. 
    On the other hand, the formula $\psi(x;u_{D}):\equiv \pi_R(x)=u_{D}$ is a $D$-still for the formula $R(x;y)$.
\end{example}

\begin{example}\label{example: division ring}
    Let $F$ be an ordered division ring. Consider the language $L:=\{+,0,(\lambda\cdot)_{\lambda\in F},<,P\}$ of ordered $F$-vector spaces with an added unary predicate $P$. Let $T_{ds}$ be the theory of dense ordered $F$-vector spaces with an order-dense proper subspace $P$. 
    Then $T_{ds}$ admits quantifier elimination by \cite[Theorem 5.15]{belegradek1999extended}, and is non-distal by \cite[Theorem 5.16]{belegradek1999extended}. 
    The formula $x-y\in P$ witnesses non-distality, but this is the only obstruction; $T_{ds}$ is distal relative to $V/P$, where $V$ denotes the home sort. By quantifier elimination and Lemma \ref{lem: relative distality in one variable enough}, it is enough to consider unary atomic formulas, i.e. formulas of the form
    \[ \overbrace{\ \lambda x=\tau(y)}^{\equiv:\varphi_1(x;y)},\ \overbrace{\lambda x<\tau(y)}^{\equiv:\varphi_2(x;y)}\text{ or }\overbrace{\tau(y)-\lambda x\in P}^{\equiv:\varphi_3(x;y)}\]
    where $\tau$ is an $L$-term and $|x|=1$. 
    The formulas $\varphi_1$ and $\varphi_2$ are distal, i.e. have $\emptyset$-stills, since for any $V\models T_{ds}$, and any $a\in V$ and $B\fin V^y$, the type
    $\tp_{\varphi_1,\varphi_2}(a/B)$ is determined by the order-type of $\lambda a$ over $\tau(B)$.
    On the other hand, $\tp_{\varphi_3}(a/B)$ is isolated by the formula $\lambda x+P=\lambda  a+P$, so the formula $\psi(x;u_{V/P}):\equiv \lambda x+P=u_{V/P}$ is a $V/P$-still for $\varphi_3$.
\end{example}

\begin{example}
    We prove in Section \ref{sec: ACVF} that the theory $\ACVF$ of algebraically closed valued fields is distal relative to the residue field sort.
\end{example} 

We show in the following lemma that in Definition \ref{def: relative distality on a set} (i), one can allow the parameter $c$ to come from $\acl(B)$ instead of $B$, as long as the algebraic formula $\chi(z;b)\in\tp(c/B)$ isolating the type is independent of $B$ and $a$. 
Similarly, the parameter $f(a,c)$ can be replaced by a parameter in $\acl\big(D(\dcl(Ba))\big)$, under a uniformity assumption.
Given a partitioned formula $\chi(z;w)$, we say that $\chi$ is \emph{uniformly algebraic} of degree $k\in\mathbb N$ in $M$ if $|\chi(M^z;b)|\leq k$ for each $b\in M^w$.
We say $\chi$ is uniformly algebraic of degree $k$ in $T$ if this is the case for every model of $T$.
Denote by $w$ a concatenation of $m$-many disjoint copies of $y$.

\begin{lemma}\label{lemma: algebraic parameters in USHD}
    Let $\varphi(x;y)$, $\psi(x;z;u_D)$,  $\chi(z;w)$, and $\chi_D(u_D;v_D)$ be $L$-formulas, where $\chi$ and $\chi_D$ are uniformly algebraic in $T$, and let $f(x;w)$ be a definable map with codomain $D^{|v_D|}$.
    Suppose for every $M\models T$ and every $a\in M^x$ and $B\fin M^y$ there is $\overline{b}\in B^m$, and some $c\in \chi(M^z;\overline{b})$ and $d\in\chi_D(M^{u_D};f(a,\overline{b}))$ with
    \begin{align}\label{eq: Lemma for algebraic USHD}
            (M,B)\models \psi(a;c;d)\ \wedge \ \forall x\big(\psi(x;c;d)\rightarrow\tp_\varphi(x/B)=\tp_\varphi(a/B)\big);
        \end{align}
    Then $\varphi(x;y)$ has a $D$-still in $T$. 
\end{lemma}
\begin{proof}
    Let $k\in\mathbb N$ be the degree of $\chi$ and let $\ell\in\mathbb N$ be the degree of $\chi_D$. 
    Let $\overline{y}$ be the concatenation of a family $(y^{i,j})_{i\leq k-1,j\leq\ell-1}$ of disjoint copies of $y$, which are disjoint from $w$ and $z$.
    Set $\psi'(x;w,\overline{y};v_D)$ to be the formula 
    \begin{align*}
        \exists z\exists u_D\Big(&\chi(z;w)\wedge\chi_D(u_D;v_D)\wedge\psi(x;z;u_D)\wedge\\
        &\forall x'\big(  \psi(x';z;u_D)\rightarrow \bigwedge_{i\leq k-1,j\leq\ell-1}(\varphi(x',y^{i,j})\leftrightarrow\varphi(x,y^{i,j}))\big)\Big). 
    \end{align*} 
    Given $a\in M^x$ and $B\fin M^y$, let $\overline{b}$, $c$, and $d$ be as in the assumption. 
    Let $c_1,\ldots,c_{k}$ enumerate $\chi(M^z;\overline{b})$ with $c_k=c$. 
    Similarly, let $d_1,\dots,d_\ell$ enumerate $\chi_D(M^{u_D};f(a,\overline b))$ with $d_\ell=d$.
    Then for each $(i,j)\in[k-1]\times [\ell-1]$ with the property that the second conjunct in (\ref{eq: Lemma for algebraic USHD}) does not hold with $c_i$ in place of $c$ and $d_j$ in place of $d$, let $e^{i,j}\in B$ witness this, i.e. 
    \[\psi(x;c_i;d_j)\not\rightarrow \big(\varphi(x,e^{i,j})\leftrightarrow\varphi(a,e^{i,j})\big).\] 
    For all other pairs $(i,j)\in[k-1]\times [\ell-1]$, fix any $e^{i,j}\in B$. 
    Then $\models \psi'(a;\overline{b},\overline{e};f(a,\overline{b}))$ and this isolates $\tp_\varphi(a/B)$.
\end{proof}

\subsection{Algebraically closed valued fields}\label{sec: ACVF}

We prove in this section that the three-sorted theory of algebraically closed valued fields is distal relative to the residue field sort.

\begin{definition}\label{definition: valued field}
    A \textit{valued field} $(K,v)$ consists of a field $K$ and a surjective map $v:K\to \Gamma\cup\{\infty\}$ where $(\Gamma,+,<)$ is an ordered abelian group and where we set $\gamma<\infty$ for each $\gamma\in \Gamma$, such that the following conditions hold for all $a$ and $b$ from $K$.
    \begin{enumerate}
        \item[(i)] $v(a\cdot b)=v(a)+v(b)$;
        \item[(ii)] $v(a)=\infty$ if and only if $a=0$;
        \item[(iii)] $v(a+b)\geq \min(v(a),v(b))$ with equality whenever $v(a)\neq v(b)$.
    \end{enumerate}
    We call $\Gamma$ the \emph{value group} and $v$ the \emph{valuation} of $(K,v)$.
    We say that $(K,v)$ is non-trivial if $v(a)>0$ for some non-zero $a\in K$.
    Condition $(iii)$ is called the \emph{ultrametric triangle inequality}.
\end{definition}

If $(K,v)$ is a valued field, let $\mathcal{O}_v(K)$ be the set of all field elements $a$ with $v(a)\geq 0$; $\mathcal{O}_v(K)$ forms a ring called the \emph{valuation ring}.
The valuation ring has a unique maximal ideal denoted $\mathfrak m_v(K)$, which coincides with the set of field elements $a$ with $v(a)>0$. 
The quotient field $\mathcal{O}_v(K)/\mathfrak m_v(K)$ is called the \emph{residue field} of the valued field $(K,v)$, which we denote by $k_v(K)$.
We write $\res_v$ for the quotient map $\mathcal{O}_v(K)\to k_v(K)$.
We will view valued fields as three-sorted structures, following notation from \cite{acvf_hils}.

\begin{definition}
    Let $L_{\Gamma ,k}$ be the three-sorted language consisting of
    \begin{itemize}
        \item a sort $\mathbf K$ together with the language $L_{\mathbf K}=\{+,\cdot,-,0,1\}$ of rings;
        \item a sort $\mathbf \Gamma$ together with the language $L_{\mathbf \Gamma}:=\{+,-,0,<,\infty\}$ and a map $\text{val}:\mathbf K\to \mathbf\Gamma$;
        \item a sort $\mathbf k$ together with the language $L_{\mathbf{k}}=\{+,\cdot,-,0,1\}$ and a map $\Res:\mathbf K^2\to \mathbf k$.
    \end{itemize}
\end{definition}

We view a valued field $(K,v)$ as an $L_{\Gamma ,k}$-structure by identifying $(K,v)$ with $(K,\Gamma,k_v(K),v,\Res_v)$, where $v$ is the interpretation of $\text{val}$, and
\[
    \Res_v(a,b):=\begin{cases}
        \res_v(\frac{a}{b}),\  \text{if }b\neq 0\text{ and }v(a)\geq v(b),\\
        0,\ \text{otherwise}.
    \end{cases}
\]
By convention, we set $\infty+\nu=\infty$ and $-\infty=\infty$ for $\nu\in\Gamma$.
Throughout, we write $(K,v)$ to denote a valued field. 
We may refer to $K$ as a valued field, and 
given an $L_{\Gamma ,k}$-formula or $L_{\Gamma ,k}$-definable map $\rho$, we write $\rho^K$ for the interpretation of $\rho$ in $(K,v)$.
We will never consider two different valuations on a given field at once, so this notation is unambiguous.

\begin{definition}
    Let $\ACVF$ be the $L_{\Gamma ,k}$-theory of algebraically closed fields with a non-trivial valuation.
\end{definition}

The remainder of this section is dedicated to proving the following theorem.

\begin{theorem}\label{thm: ACVF distal relative to k}
    $\ACVF$ is distal relative to the sort $\mathbf k$.
\end{theorem}

The following fact is used crucially in two steps of the proof. It is essentially due to Robinson \cite{robinson1956complete} and made explicit in \cite[Theorem 2.1.1]{haskell2006definable}.

\begin{fact}[Robinson]\label{fact: QE in ACVF}
    $\ACVF$ admits quantifier elimination. 
\end{fact}

By Lemma \ref{lem: relative distality in one variable enough} proving Theorem \ref{thm: ACVF distal relative to k} reduces to finding  $\mathbf{k}$-stills for all \emph{unary} partitioned formulas $\varphi(x;y)$, i.e. where $|x|=1$.
The need to work in a many-sorted language of valued fields, rather than a one-sorted one where the residue field is still present as an imaginary sort, stems from this step. 
Indeed, to apply Lemma \ref{lem: relative distality in one variable enough}, we require existence of $\mathbf k$-stills for formulas with parameters of sort $\mathbf k$.
The addition of a sort for the value group is included for the sake of generality.
Reducing further, it is sufficient by Fact \ref{fact: QE in ACVF} to prove existence of $\mathbf{k}$-stills for \emph{quantifier-free} unary partitioned formulas. 
Finally, by Lemma \ref{lem: Boolean combinations}, it suffices to find $\mathbf{k}$-stills for all \emph{atomic} unary partitioned formulas -- that is, formulas containing no quantifiers or Boolean connectives.
Note that the proof is constructive, up to the quantifier elimination from Fact \ref{fact: QE in ACVF}, which can be done by a primitive recursive algorithm by \cite{weispfenning2006quantifier}.\\

The remainder of this section is dedicated to proving that atomic unary partitioned formulas $\varphi(x;y)$ have $\mathbf{k}$-stills in $\ACVF$.
When $x$ has sort $\mathbf\Gamma$ or $\mathbf k$, this is straightforward, see Lemma \ref{lemma: relative distality of k and Gamma formulas}.
When $x$ has sort $\mathbf K$, it is enough to consider formulas of one of two forms.
Proposition \ref{prop: relative distality of valuation formulas} constructs $\mathbf{k}$-stills for formulas of the form 
\begin{align}
    &\varphi_{f,g,\sigma,\tau}(x_{\mathbf K},y_{\mathbf K},y_{\mathbf \Gamma})\ :\equiv \ v(f(x_{\mathbf K},y_{\mathbf K}))+\sigma(y_{\mathbf \Gamma})< v(g(x_{\mathbf K},y_{\mathbf K}))+\tau(y_{\mathbf \Gamma})\label{formula form 2},
\end{align}
where $f$ and $g$ are integer polynomials, and $\sigma$ and $\tau$ are $L_{\Gamma}$-terms. 
Note that this also covers formulas of
the form (\ref{formula form 2}) with $<$ replaced by $=$, since Boolean combinations of formulas with $\mathbf{k}$-stills have $\mathbf{k}$-stills.
Moreover, we get $\mathbf{k}$-stills for formulas of the form $f(x_{\mathbf K},y_{\mathbf K})=0$, since this is equivalent in $\ACVF$ to $v(f(x_{\mathbf K},y_{\mathbf K}))=v(0)$.
The proof of Proposition \ref{prop: relative distality of valuation formulas} is an adaptation of the proof of \cite[Lemma 5.5]{bays_martin}, going via a uniform Swiss cheese decomposition (Lemma \ref{lemma: swiss cheeses}) and existence of a $\mathbf k$-still for containment in valuative balls (Lemma \ref{lemma: containment has USHD}), with additional care taken to account for the value group variables. 

Proposition \ref{prop: res formulas are relatively distal} establishes existence of $\mathbf k$-stills for atomic formulas of the form 
\begin{align}
    &\varphi_{\overline{f},\overline{g},P}(x_{\mathbf K};y_{\mathbf K},y_{\mathbf k}):\equiv P\Big(\Res\big((\overline f(x_{\mathbf K},y_{\mathbf K}),\overline g(x_{\mathbf K},y_{\mathbf K})\big), y_{\mathbf k}\Big)=0,\label{formula form 3}
\end{align}
where $P$ is a polynomial with integer coefficients, and $\overline{f}$ and $\overline{g}$ are tuples of equal length of such polynomials. 
We write here $\Res(\overline{f},\overline{g})$ as shorthand for $\big(\Res(f_1,f_2),\dots,\Res(f_{|\overline{f}|},g_{|\overline{g}|})\big)$.
The proof goes via Lemma \ref{lemma: determining images in the residue field}, which establishes, roughly, that the image of a point $a$ under a finite parameterized family of maps $\big(\Res(\overline{f},\overline{g})(x;b)\big)_{b\in B}$ is determined by an instance of a formula depending only on $\overline{f}$ and $\overline{g}$ with parameters from $B$ and $\mathbf k(\dcl(Ba))$.

We will use subscripts of the form $x_{\mathbf{\Gamma}}$, $x_{\mathbf K}$, and $x_{\mathbf{k}}$, when we want to make the sort of a given variable explicit. 

\begin{lemma}\label{lemma: relative distality of k and Gamma formulas}
    Let $\varphi(x;y)$ be an atomic $L_{\Gamma ,k}$-formula where $x$ is a unary variable of the sort $\mathbf{k}$ or $\mathbf\Gamma$. Then $\varphi(x;y)$ has a $\mathbf{k}$-still in $\ACVF$.
\end{lemma}
\begin{proof}
    In case $x$ is a variable of the sort $\mathbf{k}$, the formula $\psi(x;z;u_{\mathbf k}):\equiv x=u_{\mathbf k}$, where $z$ is just an empty variable, is a $\mathbf{k}$-still for $\varphi$.
    Suppose $x$ is a variable of the sort $\Gamma$.
    Then the statement follows from the fact that $\varphi(x,y)$ is quantifier-free and that the induced structure in $\ACVF$ on $\Gamma$ is that of a divisible ordered abelian group (together with a certain distinguished element $\infty$), and this is distal, by quantifier-elimination and distality of o-minimal structures.
\end{proof}

It remains to prove that any atomic $L_{\Gamma ,k}$-formula $\varphi(x_{\mathbf K};y)$, where $x_{\mathbf K}$ is a unary variable of the sort $\mathbf{K}$, has a $\mathbf{k}$-still in $\ACVF$.
As described above, this reduces to two cases, (\ref{formula form 2}) and (\ref{formula form 3}).
We begin by proving relative distality for formulas of the form $\varphi_{f,g,\sigma,\tau}$ as in \eqref{formula form 2}.
For this, we fix multi-variables $y_{\mathbf K}$ and $y_{\mathbf \Gamma}$ and a unary variable $x_{\mathbf K}$.
We fix polynomials $f(x_{\mathbf K},y_{\mathbf K})$ and $g(x_{\mathbf K},y_{\mathbf K})$ in $\mathbb Z[x_{\mathbf K},y_{\mathbf K}]$ and $L_\Gamma$-terms $\sigma(y_{\mathbf \Gamma})$ and $\tau(y_{\mathbf \Gamma})$.
\begin{prop}\label{prop: relative distality of valuation formulas}
    Formulas of the form $\varphi_{f,g,\sigma,\tau}(x_{\mathbf{K}};y_{\mathbf{K}},y_{\mathbf \Gamma})$ have $\mathbf k$-stills in $\ACVF$.
\end{prop}

To prove Proposition \ref{prop: relative distality of valuation formulas}, it will be useful to introduce the notion of valuative balls, which we just refer to as \emph{balls}.
Given $c\in K$ and $\nu\in \Gamma$, we call the set $b^{o}(c,\nu):=\{x\in K:v(x-c)>\nu\}$ the \emph{open ball of radius $\nu$ around $c$}. 
Similarly, we call $b^{f}(c,\nu):=\{x\in K:v(x-c)\geq \nu\}$ the \emph{closed ball of radius $\nu$ around $c$}.
Note that having \emph{large} valuation is analogous to having \emph{small} value. Thus, as the \emph{valuative} radius of a ball increases, the ball shrinks.
We write $x\in b^o(y_{\mathbf K},y_{\mathbf \Gamma})$ as shorthand for the formula $v(x-y_{\mathbf K})>y_{\mathbf \Gamma}$, and similarly for $x\in b^f(y_{\mathbf K},y_{\mathbf \Gamma})$.
When $c\in K$ and $d\in K$, we may write $b^o(c,d)$ in place of $b^o(c,v(d))$, and similarly for $b^f(c,d)$.

We record some useful facts about the geometry of valuative balls, which -- as a result of ultrametricity -- behave somewhat unintuitively. 
 For balls $\beta_1$ and $\beta_2$, write $\beta_1\sqcup\beta_2$ for the \emph{join} of $\beta_1$ and $\beta_2$, that is the smallest ball containing $\beta_1$ and $\beta_2$.

\begin{lemma}\label{lemma: facts about balls}
    Let $\beta_1,\ldots,\beta_n$ be balls in a valued field $(K,v)$ and let $\star\in\{o,f\}$.
    \begin{enumerate}
        \item[(i)] Balls cannot intersect non-trivially, i.e. if $\beta_1\cap\beta_2\neq\emptyset$ then $\beta_1\subseteq \beta_2$ or $\beta_2\subseteq \beta_1$.
        
        \item[(ii)] Every point in a ball is a center, i.e. if $c'\in b^{\star}(c,\nu)$ then $b^{\star}(c,\nu)=b^{\star}(c',\nu)$.

        \item[(iii)] Distances between balls are well-defined, i.e. if $\beta_1\cap\beta_2=\emptyset$ then $v(c_1-c_2)=v(c_1'-c_2')=:v(\beta_1-\beta_2)$ for any $c_1,c_1'\in \beta_1$ and $c_2,c_2'\in\beta_2$.
        \item[(iv)]  $\beta_1\sqcup\ldots\sqcup \beta_n$ is either $\beta_i$ or $\beta_i\sqcup\beta_j=b^{f}(c,v(\beta_i-\beta_j))$ for some $i,j\leq n$ and any $c\in\beta_i\cup\beta_j$.
    \end{enumerate}
\end{lemma}
\begin{proof}
    (i) follows from the ultrametric triangle inequality and immediately implies (ii). For (iii), by the ultrametric triangle inequality,
    \[v(c_i-c_j)=v(c_i-c_i'+(c_i'-c_j')+c_j'-c_j)=\min(v(c_i-c_i'),v(c_i'-c_j'),v(c_j'-c_j))=v(c_i'-c_j').\]
    Finally, for (iv), if some $\beta_i$ contains all $\beta_j$, then $\beta_1\sqcup\dots\sqcup\beta_n=\beta_i$. Otherwise, we may assume by (i) that $\beta_1,\ldots,\beta_n$ are disjoint. Find $i,j\leq n$ such that $v(\beta_i-\beta_j)$ is minimal. Then $\beta_1\cup\ldots\cup\beta_n\subseteq \beta_i\sqcup \beta_j$ by the ultrametric triangle inequality, so $\beta_i\sqcup\ldots\sqcup\beta_n\subseteq\beta_i\sqcup\beta_j$. On the other hand, clearly $\beta_i\sqcup\beta_j\subseteq\beta_1\sqcup\ldots\sqcup\beta_n$.
\end{proof}

\begin{lemma}\label{lemma: swiss cheeses}
    There exist $L_{\Gamma ,k}$-formulas $\chi_1(u_{\mathbf K};y_{\mathbf K})$ and $\chi_2(u_{\Gamma};y_{\mathbf K},y_{\mathbf \Gamma})$, uniformly algebraic in $\ACVF$, such that for any $(K,v)\models \ACVF$ and $(b,\mu)\in K^{|y_{\mathbf K}|}\times \Gamma^{|y_{\mathbf \Gamma}|}$, the set $\varphi_{f,g,\sigma,\tau}(K;b,\mu)$ is a Boolean combination of balls
    with center in $\chi_1(K;b)$ and radius in $\chi_2(\Gamma;b,\mu)$.
\end{lemma}
\begin{proof}
    Fix $(K,v)\models\ACVF$ and write $\Gamma:=\mathbf \Gamma(K)$, and fix $(b,\mu)\in K^{|y_{\mathbf K}|}\times \Gamma(K)^{|y_{\mathbf \Gamma}|}$.
    We make a case distinction on $K$, $b$, and $\mu$, and prove the statement in each case separately.
    This is sufficient, since finitely many uniformly algebraic formulas can be coded into one uniformly algebraic formula.
    Recall that
    \[a\in \varphi_{f,g,\sigma,\tau}(K;b,\mu)\iff v(f(a,b))+\sigma(\mu)<v(g(a,b))+\tau(\mu).\]
    It follows that if $\sigma(\mu)=\infty$ or $f(x;b)$ is the constant polynomial $0$, then $\varphi(K;b,\mu)=\emptyset$, and the statement follows since $\emptyset=b^o(0,\infty)$.
    Otherwise, if $\tau(\mu)=\infty$ or $g(x;b)$ is the constant polynomial $0$, then $\varphi(K;b,\mu)=K\setminus Z_{f(x;b)}(K)$ where $Z_{f(x;b)}(K)$ is the set of roots of $f(x;b)$. 
    Again, the statement follows since $\{\gamma\}=b^f(\gamma,\infty)$ whenever $\gamma\in K$.
   
    Henceforth, we assume that $\sigma^K(\mu)\neq \infty$ and $\tau^K(\mu)\neq\infty$, and $f(x;b)$ and $g(x;b)$ are not zero everywhere. In particular, we may assume that $\sigma=0$ by replacing $\tau$ with $\tau-\sigma$.
    Let $\gamma_1,\dots,\gamma_m\in K$ be the roots of $f(x;b)\cdot g(x;b)$ and let $t_1(y)$ and $t_2(y)$ be $L_{\text{ring}}$-terms such that
    \[f(x,b)=t_1(b)(x-\gamma_1)^{n_1}\cdot\ldots\cdot(x-\gamma_m)^{n_m}\text{ and } g(x,b)=t_2(b)(x-\gamma_{1})^{n'_{1}}\cdot\ldots\cdot(x-\gamma_m)^{n'_m}\]
    for positive integers $n_i,n_i'\leq\max(\deg_x(f),\deg_x(g))$.
    Note that $t_1(b),t_2(b)\neq 0$ since $f(x;b)$ and $g(x;b)$ are non-zero.
    By Definition \ref{definition: valued field} (i), we have
    \begin{align}\label{eq: linear factors}
        a\in \varphi_{f,g,\sigma,\tau}(x;b,\mu)\iff \sum_{i=1}^{m}\ell_iv(a-\gamma_i)<v\big(\frac{t_2}{t_1}(b)\big)+\tau(\mu),
    \end{align}
    where $\ell_i:=n_i-n_i'$.
    The following claim is strengthened to allow for an inductive proof.
    We use throughout that since $(K,v)$ is an algebraically closed valued field, the group $\Gamma$ is divisible.\\
    
    \underline{Claim:} For $N\in\mathbb N$, $\nu\in v(K)$, $c_1,\ldots,c_N\in K$, and $k_1,\ldots,k_N\in\mathbb Z$, the subset of $K$ defined by $\sum_{i\leq N}k_iv(x-c_i)<\nu$ is equivalent to a Boolean combination of $\emptyset$ and balls $b^\star(c,\eta)$ where $c= c_i$ for some $i\leq N$ and either $\eta=v(c_i-c_j)$ for some $j\leq N$ or 
    \[\eta=\frac{1}{\sum_{i\in I_0}k_i}\cdot\Big(\nu-\sum_{j\in[N]\setminus I_0}k_jv(c_{p(j)}-c_{q(j)})\Big)\]
    for a non-empty set $I_0\subseteq [N]$ with $\sum_{i\in I_0}k_i\neq 0$ and maps $p,q:[N]\setminus I_0\to[N]\setminus I_0$ with $p(i)\neq q(i)$ for each $i$.
    \begin{proof}[Proof of Claim]
        The base case $N=1$ is clear: If $k_1=0$, then the statement is equivalent to $0<\nu$, so the set is either $\emptyset$ or $K$.
        If $k_1>0$, then the set is the complement of the closed ball with center $c_1$ and radius $\frac{\nu}{k_1}$.
        If $k_1<0$, the set is the open ball with center $c_1$ and radius $\frac{\nu}{k_1}$.
       
        We consider the formula $\psi(x):\equiv  \sum_{i\leq N+1}k_iv(x-c_i)<\nu$. We consider a set $X$ fixing a specific order-type of $v(x-c_{N+1})$ over $\{v(c_1-c_{N+1}),\ldots,v(c_N-c_{N+1})\}$ and prove that $\psi(x)$ is a Boolean combination of the claimed form, on the intersection with $X$.
        This suffices, since $X$ is itself a Boolean combination as in the claim. 

    \underline{Case 1:} If $v(x-c_{N+1})=v(c_N-c_{N+1})$, then $\psi(x)$ is equivalent to $\sum_{i\leq N}k_iv(x-c_i)<\nu - k_{N+1}v(c_N-c_{N+1})$ on $X$.

    \underline{Case 2:} If $v(x-c_{N+1})<v(c_N-c_{N+1})$, then $v(x-c_N)=v(x-c_{N+1})$, and so $\psi(x)$ is equivalent to $\sum_{ i< N}k_iv(x-c_i)+(k_N+k_{N+1})v(x-c_{N})<\nu $.

    \underline{Case 3:} If $v(x-c_{N+1})>v(c_N-c_{N+1})$, then $v(x-c_N)=v(c_N-c_{N+1})$, so $\psi(x)$ is equivalent to $\sum_{i<N}k_iv(x-c_i)+k_{N+1}v(x-c_{N+1})<\nu - k_Nv(c_N-c_{N+1})$.

    In each case, the statement follows by the inductive hypothesis.
    \end{proof}
    Applying the claim to (\ref{eq: linear factors}),
    we see that $\varphi_{f,g,\sigma,\tau}(x,b)$ is equivalent to a Boolean combination of balls with a center $\gamma_i$ and with radius of the form $v(\gamma_i-\gamma_j)$ or of the form
    \begin{equation}\label{eq: radius formula}
        \frac{1}{\sum_{i\in I_0}\ell_i}\cdot \left(v\left(\frac{t_1(b)}{t_2(b)}\right)+\tau(\mu)-\sum_{j\in[m]\setminus I_0}\ell_jv\big(\gamma_{p(j)}-\gamma_{q(j)}\big)\right),
    \end{equation}
    where $I_0$, $p$, and $q$ are as in the claim.
    Set $d:=\deg_x(f(x,y))+\deg_x(g(x,y))$. 
    Note that in (\ref{eq: radius formula}), we have $m\in[d]$, $\ell_i\in\{-d,\dots,d\}$, $I_0\subseteq [m]$ with $\sum_{i\in I_0}\ell_i\neq 0$, and $p,q:[m]\setminus I_0\to [m]\setminus I_0$ with $p(i)\neq q(i)$ for each $i$.
    Let $\Delta$ be the set of all tuples $(m,\overline{\ell},I_0,p,q)$ which satisfy these properties.
    Note that $\Delta$ is finite.
    For $\delta=(m,\overline{\ell},I_0,p,q)\in \Delta$, let $s_\delta:={\sum_{i\in I_0}\ell_i}$ and let 
    \[P^+_\delta(\gamma_1,\dots,\gamma_m):=\prod_{j\notin I_0, \ell_j>0}(\gamma_{p(j)}-\gamma_{q(j)})^{\ell_j}\text{ and }P^-_\delta(\gamma_1,\dots,\gamma_m):=\prod_{j\notin I_0, \ell_j<0}(\gamma_{p(j)}-\gamma_{q(j)})^{-\ell_j}.\]
    By additivity of $v$ (Definition \ref{definition: valued field} (i)), the term (\ref{eq: radius formula}) is equal to 
    \begin{equation}\label{eq: simplified radius formula}
        Q_\delta(b,\mu,\gamma_1,\dots,\gamma_m):=\frac{1}{s_\delta}\left(\tau(\mu)+v\left(\frac{t_1(b)P_\delta^-(\gamma_1,\dots,\gamma_m)}{t_2(b)P_\delta^+(\gamma_1,\dots,\gamma_m)}\right)\right).
    \end{equation}
    For each $\delta=(m,\overline{\ell},I_0,p,q)\in\Delta$, let $\chi_2^{\delta}(u_{\mathbf \Gamma};y_{\mathbf K},y_{\mathbf \Gamma},z^1_{\mathbf K},\dots,z^m_{\mathbf K})$ be the $L_{\Gamma ,k}$-formula expressing that $u_{\mathbf \Gamma}=Q_\delta(y_{\mathbf K},y_{\mathbf \Gamma},z_{\mathbf K}^1,\dots,z_{\mathbf K}^m)$ , i.e.
    \[
        \chi_2^{\delta}(u_{\mathbf \Gamma},y_{\mathbf K},y_{\mathbf{\Gamma}},\overline{z}_{\mathbf K}):\equiv\  s_\delta u_{\mathbf \Gamma}+v(t_2(y_{\mathbf K}))+v(P_\delta^+(\overline{z}_{\mathbf K}))-\tau(y_{\mathbf \Gamma})-v(t_1(y_{\mathbf K}))-v(P^-_\delta(\overline{z}_{\mathbf K}))=0. \]

    Now, let $\chi_{1}(u_{\mathbf K},y_{\mathbf K})$ be the $L_{\text{ring}}$-formula
    \[f(u_{\mathbf K},y_{\mathbf K})\cdot g(u_{\mathbf K},y_{\mathbf K})=0\wedge \neg\forall x_{\mathbf K}(f(x_{\mathbf K},y_{\mathbf K})\cdot g(x_{\mathbf K},y_{\mathbf K})=0).\] 
    Then each $\gamma_i$ is a solution to $\chi_1(u_{\mathbf K};b)$ and $\chi_1(u_{\mathbf K};y_{\mathbf K})$ is uniformly algebraic (of degree $\leq d$) in $\ACVF$.
    Finally, we set $\chi_2(u_{\mathbf \Gamma},y_{\mathbf K},y_{\mathbf \Gamma})$ to be the disjunction of the finitely many formulas 
    \begin{itemize}
        \item $\exists z_{\mathbf K}^1,z_{\mathbf K}^2\big(\chi_1(z_{\mathbf K}^1,y_{\mathbf K})\wedge\chi_1(z_{\mathbf K}^2,y_{\mathbf K})\wedge u_{\mathbf \Gamma}=v(z^1_{\mathbf K}-z^2_{\mathbf K})\big)$,
        \item $\exists z_{\mathbf K}^1\dots z_{\mathbf K}^m\Big(\bigwedge_{i=1}^m\chi_1(z^i_{\mathbf K},y_{\mathbf K})\ \wedge\  \bigwedge_{i<j\leq m} z^i_{\mathbf K}\neq z^j_{\mathbf K}\ \wedge\ \tau(y_{\mathbf \Gamma})\neq\infty\wedge\   \chi_2^{\delta}(u_{\mathbf \Gamma},y_{\mathbf K},y_{\mathbf \Gamma},\overline{z}_{\mathbf K})\Big)$,
    \end{itemize}
    with $\delta$ ranging over $\Delta$.
    Together with the claim above, this yields the desired result, since $v(\gamma_i-\gamma_j)\in\chi_2(\Gamma;b,\mu)$ and $Q_\delta(b,\mu,\gamma_1,\dots,\gamma_m)\in\chi_2(\Gamma;b,\mu)$ whenever $i\neq j$ and $\delta\in\Delta$. 
    Moreover, $\chi_2(u_{\mathbf \Gamma};y_{\mathbf K},y_{\mathbf\Gamma})$ is uniformly algebraic, since $\chi_2^\delta(u_{\mathbf \Gamma};y_{\mathbf K},y_{\mathbf \Gamma},\overline{z}_{\mathbf K})$ has a unique solution whenever $\overline{z}_{\mathbf K}\subseteq \chi_1(K;y_{\mathbf K})$ and $\tau(y_{\mathbf \Gamma})\neq\infty$.
\end{proof}


\begin{lemma}\label{lemma: containment has USHD}
    The formulas $x_{\mathbf K}\in b^{o}(y_{\mathbf K},y_{\mathbf \Gamma})$ and $x_{\mathbf K}\in b^{f}(y_{\mathbf K},y_{\mathbf \Gamma})$ have $\mathbf k$-stills in $\ACVF$. 
\end{lemma}
\begin{proof}
    Let $(K,\Gamma,k)\models\ACVF$ and fix a finite set $B\subseteq K\times \Gamma$ and a point $a\in K$. 
    Write $\mathcal{B}_0$ for the collection of open and closed balls $b^{\star}(b)$ with $b\in B$. 

    Denote by $\mathcal{B}$ the closure of $\mathcal{B}_0$ under joins. 
    Thus, by Lemma \ref{lemma: facts about balls} (iv), each ball in $\mathcal{B}$ is either in $\mathcal{B}_0$ or of the form $b^{f}(b_K,b'_K-b_K)$ for some $b,b'\in B$.
    We will use this fact repeatedly in the arguments below.
    For notational convenience, we write $x$ for $x_\mathbf{K}$, and we let $p(x)$ denote the ``containment-type" of $\tp_\in(a/\mathcal{B})$.\footnote{Formally, to make sense of this expression, we add an imaginary sort $\mathbf B$ for open and closed balls and define a suitable relation $\in$ on $\mathbf K\times \mathbf B$. This is explained in Remark \ref{remark: about containment and k-stills} (iii).}
    \\
    We make a case distinction for $a$ and $B$ and prove that in each case, $p(x)$ is isolated by an instance of a formula $\psi(x;c;d)$ where $c$ is a tuple in $B$ and $d$ is of the form $f(a,c)$ for some definable map $f$ with codomain $\mathbf k^{|d|}$.
    Importantly, $\psi$ and $f$ will not depend on $a$ or $B$, but just the case.
    This gives a piecewise $\mathbf{k}$-still for both formulas $x_{\mathbf K}\in b^{o}(y_{\mathbf K},y_{\mathbf \Gamma})$ and $x_{\mathbf K}\in b^{f}(y_{\mathbf K},y_{\mathbf \Gamma})$, which can be coded  into a single $\mathbf{k}$-still as described in Remark \ref{remark: facts about relative USHDs} (iii).\newline

    \underline{Case 1:} $a\notin \beta$ for any $\beta\in \mathcal{B}$.
    Then let $\beta^1$ be the join of all balls in $\mathcal{B}$. 
    The assertion $x\notin \beta^1$ already isolates $p(x)$, and can be written as a formula of one of the forms
    \[\neg\big(x\in b^o(b^1)\big),\ \neg\big(x\in b^f(b^1)\big),\ \neg\big(x\in b^f\big(b^1_K,b^2_K-b^1_K\big)\big),\]
    with $b^1,b^2\in B$.
    \\

    Note that if $a$ is contained in some ball in $\mathcal{B}$, then there is a unique minimal one that it is contained in, by Lemma \ref{lemma: facts about balls} (i).\newline 
    
    \underline{Case 2:} Suppose the minimal ball $\beta^1\in\mathcal{B}$ with $a\in\beta^1$ cannot be written as $\beta\sqcup\beta'$ for distinct $\beta,\beta'\in\mathcal{B}_0$.

    Then, if we let $\beta^2$ be the join of all balls in $\mathcal{B}_0$ which are strictly contained in $\beta^1$, it follows that $p(x)$ is isolated by the assertion $x\in \beta^1\wedge x\notin \beta^2$.
    Now, depending on the form of $\beta^1$ and $\beta^2$, this expression is equivalent to a formula of one of the forms
    \[\big(x\in b^{\star}(b^1)\wedge x\notin b^{\bullet}(b^2)\big),\ \big(x\in b^\star(b^1)\wedge x\notin b^f\big(b^2_K,b^3_K-b^2_K\big)\big),\]
    with $b^1,b^2,b^3\in B$ and $\star,\bullet\in\{o,f\}$.\\

    \underline{Case 3:} Suppose the minimal ball $\beta^1\in \mathcal{B}$ containing $a$ is $\beta\sqcup\beta'$ for distinct $\beta,\beta'\in\mathcal{B}_0$.

    Then let $b,b'\in B$ be the tuples defining $\beta$ and $\beta'$, so that $\beta^1=b^f(b_K,b_K'-b_K)$. 

    \underline{Case 3.1:}
    Suppose first that $a$ is contained in a ball of the form $\beta^*=b^o(b_K^*,b_K-b'_K)$ for some $b^*\in B$ with $b^*_K\in \beta^1$. 
    Then, either $v(b_K^*-b_K)=v(b_K'-b_K)$ or $v(b_K^*-b_K')=v(b_K'-b_K)$, by the ultrametric triangle inequality. 
    Assume without loss of generality that the former holds, so that $\beta^*=b^o(b_K^*,b_K-b_K^*)$.
    Next, find the unique maximal ball $\beta'$ in $\mathcal{B}(B)$ which is strictly contained in $\beta^*$.
    This ball exists, since $\mathcal{B}(B)$ is closed under joins and $\beta^*$ is open. 
    Now, $p(x)$ is isolated by the assertion $x\in\beta^*\wedge x\notin \beta'$. 
    Since any center of $\beta'$ is a center of $\beta^*$ by Lemma \ref{lemma: facts about balls} (ii), this assertion is equivalent to a formula of the form
    \[\big(x\in b^o(b_K^*,b_K-b_K^*)\wedge x\notin b^\star(b^2)\big),\ \big(x\in b^o(b_K^*,b_K-b_K^*)\wedge x\notin b^f(b^*_K,b^2_K-b_K^*)\big),\]
    with $\star\in\{o,f\}$ and $b^2\in B$.

    \underline{Case 3.2:}
    Suppose now that $a$ is not contained in any ball of the form $b^o\big(b_K^*,b_K-b_K')$ with $b^*\in B$ and $b^*_K\in \beta^1$.
    Then $p(x)$ is isolated by the formula
    \[\res\Big(\frac{x-b_K}{b_K'-b_K}\Big)=\alpha,\]
    where $\alpha:=\res\Big(\frac{a-b_K}{b_K'-b_K}\Big)$.
    Indeed, if $a'\in K$ satisfies the above formula, then $v(a'-a)>v(b'_K-b_K)$, so in particular, $a'\in b^f(b_K,b'_K-b_K)$ and moreover, $a'\notin b^o(b^*_K,b_K-b'_K)$ for any $b^*\in B$ with $b^*_K\in \beta^1$.
\end{proof}

\begin{remark}\label{remark: about containment and k-stills}
\begin{enumerate}

    \item[(i)] Lemma \ref{lemma: containment has USHD} holds in fact in the theory of all valued fields, in place of $\ACVF$, by the same argument.
    The proof does not use algebraic closedness.

    \item[(ii)] Given a valued field $(F,v)$ with finite residue field and the collection $\mathcal{B}$ of all valuative balls in $F$, the $\mathbf k$-still coming from Lemma \ref{lemma: containment has USHD} induces a DCD $\mathcal{T}$ of exponent $\leq 3$ for the containment-relation $E\subseteq F\times \mathcal{B}$.
    This is because the case distinction in the above proof gives a piecewise $\mathbf{k}$-still of arity three.
    Now $\mathcal{T}$ can be modified by removing unnecessary cells in order to obtain a DCD of exponent two of $E$.
    Indeed, the only part where three parameters appear is the second formula defined in Case 2 and Case 3 of the proof. 
    However, given a finite set $\mathcal{B}_0$ of balls, after fixing a ball $\beta^1\in\mathcal{B}_0$, a cell of the form $\beta^1\setminus(\beta^2\sqcup\beta^3)$ should be included in $\mathcal{T}(B)$ only if $\beta^2,\beta^3\in\mathcal{B}_0$ are such that $\beta^2\sqcup\beta^3$ is the join of all balls in $\mathcal{B}_0$ which are strictly contained in $\beta^1$. 
    Similarly, after fixing two balls $\beta^1$ and $\beta^3$ in $\mathcal{B}_0$, a cell of the form $(\beta^1\sqcup\beta^2)\setminus\beta^3$ or $(\beta^1\sqcup\beta^2)\setminus(\beta^3\sqcup\beta^1)$ should only be included if $\beta^3$ or $\beta^3\sqcup\beta^1$, respectively, is the join of all balls in $\mathcal{B}_0$ which are strictly contained in $\beta^1\sqcup\beta^2$.

    \item[(iii)] Lemma \ref{lemma: containment has USHD} is still valid at the level of balls; if we let $\acvfball$ be the expansion of $\ACVF$ by the imaginary sort of all balls, then the proof of the lemma shows that the formula $x\in u_{\text{bl}}$, with $u_{\text{bl}}$ of the ball-sort, has a $\mathbf{k}$-still in $\acvfball$.
    Firstly, the join-operation on the ball-sort is defined, by Lemma \ref{lemma: facts about balls} (iv).
    Secondly, if $\beta_1$ and $\beta_2$ are non-intersecting balls and $a$ is a point, then for any points $b_1,b_1'\in\beta_1$ and $b_2,b_2'\in\beta_2$, we get
    \[\res\big(\frac{a-b_1}{b_2-b_1}\big)=\res\big(\frac{a-b_1'}{b_2'-b_1'}\big),\]
    so the expression $\res(\frac{a-\beta_1}{\beta_2-\beta_1})$ is well-defined.
    Indeed, this follows from the fact that
    \[v\big(\frac{b_2-b_1}{b_2'-b_1}-1\big)=v\big(\frac{b_2-b_2'}{b_2'-b_1}\big)>0\]
    and so 
    \[\res\big(\frac{b_2-b_1}{b_2'-b_1}\big)=\res\big(\frac{b_2-b_1}{b_2-b_1'}\big)=1.\]
\end{enumerate}
\end{remark}

We are now ready to prove Proposition \ref{prop: relative distality of valuation formulas}, stating that formulas of the form $\varphi_{f,g,\sigma,\tau}$ as defined in (\ref{formula form 2}) have $\mathbf{k}$-stills in $\ACVF$.

\begin{proof}[Proof of Proposition \ref{prop: relative distality of valuation formulas}]
    Write $\varphi:\equiv \varphi_{f,g,\sigma,\tau}$.
    Let $(K,\Gamma,k)\models \ACVF$
    and write $y=(y_{\mathbf K},y_{\mathbf \Gamma})$.
    By Lemma \ref{lemma: swiss cheeses}, there are formulas $\xi_1(s_{\mathbf K},y_{\mathbf K})$ and $\xi_2(s_{\mathbf \Gamma},y)$, both uniformly algebraic in $\ACVF$, such that for any $a\in K$ and $B\fin K^y$, the type $\tp_{\varphi}(a/B)$ is isolated by the containment-type of $a$ over the set of all open and closed balls with a center  $c$ in $\xi_1(K;b_{\mathbf K})$ and radius $\gamma$ in $\xi_2(\Gamma;b)$ for some $b:=(b_{\mathbf K},b_{\mathbf{\Gamma}})\in B$.

    By Lemma \ref{lemma: containment has USHD}, there is a formula $\psi(x;z;u_{\mathbf k})$, where $z$ is a concatenation of disjoint copies of $(s_{\mathbf K},s_{\mathbf \Gamma})$, which do not depend on $a$ and $B$, such that for some tuple $f$ from $\bigcup_{b\in B}(\xi_1(K;b_{\mathbf K})\cup\xi_2(\Gamma;b))=:\xi(B)$ and some $\alpha\in \mathbf{k}(\dcl(\xi(B)a)$, the formula $\psi(x;f;\alpha)$ isolates this containment-type, and therefore isolates $\tp_\varphi(a/B)$.
    More specifically, the proof of Lemma \ref{lemma: containment has USHD} shows that $\alpha$ is of the form $\res_v\big(\frac{a-c}{c'-c}\big)$ for some $c\in\xi_1(K;b_{\mathbf K})$ and $c'\in\xi_1(K;b'_{\mathbf K})$ with $b,b'\in B$.
    Therefore, defining $\chi(u_{\mathbf{k}};x,y,y')$ to be the formula
    \[\exists s_{\mathbf{K}},s_{\mathbf{K}}'\big(\xi_1(s_{\mathbf K},y_{\mathbf K})\wedge\xi_1(s_{\mathbf K}',y_{\mathbf K}')\wedge u_{\mathbf{k}}=\Res(x-s_{\mathbf{K}},s_{\mathbf{K}}'-s_{\mathbf K})\big),\]
    we see that $\alpha\in\chi(\mathbf k(K);a,b,b')$.
    By quantifier elimination (Fact \ref{fact: QE in ACVF}), there is a formula $\chi_{\mathbf k}(u_{\mathbf{k}};w_{\mathbf{k}})$ and a definable map $\rho(x,y,y')$ with codomain $\mathbf{k}^{|w_{\mathbf k}|}$ such that 
    \[\ACVF\vdash \chi(u_{\mathbf k};x,y,y')\leftrightarrow \chi_{\mathbf k}(u_{\mathbf{k}},\rho(x,y,y')).\]
   Thus, Lemma \ref{lemma: algebraic parameters in USHD} applies, and we conclude that $\varphi(x;y_\mathbf{K},y_{\mathbf \Gamma})$ has a $\mathbf{k}$-still in $\ACVF$.
\end{proof}

We now turn to formulas of the form $\varphi_{\overline{f},\overline{g},P}(x_{\mathbf K};y_{\mathbf K},y_{\mathbf k})$, as defined in \eqref{formula form 3}. 
We fix tuples  $\overline{f}=(f_1,\dots,f_N)$ and $\overline{g}=(g_1,\dots,g_N)$ of polynomials in $\mathbb Z[x_{\mathbf K},y_{\mathbf K}]$ and a polynomial $P(u_{\mathbf k})$ in $\mathbb Z[u_{\mathbf k}]$, where $|u_{\mathbf k}|=N+|y_{\mathbf k}|$.
Recall that the formula $\varphi_{\overline{f},\overline{g},P}(x_{\mathbf K};y_{\mathbf K},y_{\mathbf k})$ is defined as 
\[P\Big(\Res\big(\overline{f},\overline{g}\big)(x_{\mathbf K},y_{\mathbf K}),y_{\mathbf k}\Big)=0.\]

\begin{prop}\label{prop: res formulas are relatively distal}
    The formula $\varphi_{\overline{f},\overline{g},P}(x_{\mathbf K};y_{\mathbf K},y_{\mathbf k})$ as in (\ref{formula form 3}) has a $\mathbf k$-still in $\ACVF$.
\end{prop}

For notational convenience, write $x$ for $x_{\mathbf K}$.
Fix two polynomials $f(x,y_{\mathbf K})$ and $g(x,y_{\mathbf K})$ with integer coefficients. 
We write 
$\rho_{f,g}(x,y_{\mathbf K})$ for the definable map $\Res(f(x_,y_{\mathbf K}),g(x,y_{\mathbf K}))$.
For $b\in K^{|y_{\mathbf K}|}$, we write $Z_{f(x;b)}(K)$ for the set of roots of $f(x;b)$ in $K$.

\begin{lemma}\label{lemma: determining images under the residue map}
    Let $(K,v)\models\ACVF$ and $a,a'\in K$ and $b\in K^{|y_{\mathbf K}|}$. 
    Suppose $v(a'-a)>v(c-a)$ holds for each $c\in Z_{f(x;b)}(K)\cup Z_{g(x;b)}(K)$. 
    Then $\rho_{f,g}^K(a,b)=\rho_{f,g}^K(a',b)$.
\end{lemma}
\begin{proof}
    Since $K$ is algebraically closed, the polynomials $f(x;b)$ and $g(x;b)$ decompose into linear factors, say 
    \[f(x,b)=e_1\cdot (x-c_1)\cdot\ldots\cdot(x-c_m),\ g(x,b)=e_2\cdot (x-d_{1})\cdot\ldots\cdot(x-d_n).\]
    Let $C:=\{c_1,\dots,c_m,d_1,\dots,d_n\}$.
    By assumption, for any $c\in C$,
    \[v\Big(\frac{a'-c}{a-c}-1\Big)=v\Big(\frac{a'-a}{a-c}\Big)>0,\]
    and so $\res_v\big(\frac{a'-c}{a-c}\big)=1$ and similarly, $\res_v\big(\frac{a-c}{a'-c}\big)=1$.
    Moreover, we get $v(a-c)=v(a'-c)$ for all $c\in C$ by the ultrametric triangle inequality.
    In particular, if $v\big(f(a,b)\big)\neq v\big(g(a,b)\big)$ then $v\big(f(a',b)\big)\neq v\big(g(a',b)\big)$ and if $g(a,b)=0$ then $g(a',b)=0$. 
    Thus,
    \[\Res_v(f(a,b),g(a,b))=0\implies \Res_v(f(a',b),g(a',b))=0.\]
    Suppose now $v\big(f(a,b)\big)= v\big(g(a,b)\big)\neq\infty$. 
    Then by the above,
    \begin{align*}
        \Res_v\big(f(a',b),g(a',b)\big)&=\res_v\Big(\frac{e_1(a'-c_1)\cdot\ldots\cdot(a'-c_m)}{e_2(a'-d_1)\cdot\ldots\cdot(a'-d_n)}\Big)\\
        &=\res_v\Big(\frac{f(a,b)}{g(a,b)}\cdot \prod_{i=1}^m\frac{a'-c_i}{a-c_i}\cdot \prod_{j=1}^n\frac{a-d_j}{a'-d_j}\Big)\\
        &=\res_v\Big(\frac{f(a,b)}{g(a,b)}\Big)\cdot\prod_{i=1}^m\res_v\Big(\frac{a'-c_i}{a-c_i}\Big)\cdot\prod_{j=1}^n\res_v\Big(\frac{a-d_j}{a'-d_j}\Big)\\
        &=\res_v\Big(\frac{f(a,b)}{g(a,b)}\Big)=\Res_v\big(f(a,b),g(a,b)\big).
    \end{align*}
\end{proof}

Recall $\overline{f}=(f_1,\dots,f_N)$ and $\overline{g}=(g_1,\dots,g_N)$.
For $B\fin K^{|y_{\mathbf K}|}$, we write
\[Z_{\overline{f},\overline{g}}(K;B):=\bigcup_{b\in B,\ell\leq N}Z_{f_\ell(x;b)}(K)\cup Z_{g_\ell(x;b)}(K).\]

In the second and final step towards showing Proposition \ref{prop: res formulas are relatively distal}, we turn now to proving that given $a\in K$ and a finite set $B\subseteq K^{|y_K|}$, the image of $a$ under the tuple of maps $\big(\Res(\overline{f},\overline{g})(x,b)\big)_{b\in B}$ is already determined by the image of $a$ under a map $\pi(x;c)$ into $\mathbf k$ and a formula $\psi(x_{\mathbf K};c)$ where $c$ may come from $B$ and $Z_{\overline{f},\overline{g}}(K;B)$, and $\pi$ and $\psi$ only depend on $\overline{f}$ and $\overline{g}$.
Recall that $x$ is a unary variable of the sort $\mathbf K$.

\begin{lemma}\label{lemma: determining images in the residue field}
    There exists an $L_{\Gamma ,k}$-formula $\psi(x;z_{\mathbf K})$ and an $L_{\Gamma ,k}$-definable map $\pi(x;z_{\mathbf K})$ into a finite power of $\mathbf{k}$, such that for any $(K,v)\models\ACVF$, $a\in K$, and $B\fin K^{|y_K|}$, there is 
    a tuple $c$ from $\bigcup B\cup Z_{\overline{f},\overline{g}}(K;B)$
    such that $\models \psi(a;c)$, and for any $a'\in\psi(K;c)$ with $\pi(a';c)=\pi(a;c)$, we have
    $\Res^K(\overline{f},\overline{g})(a',b)=\Res^K(\overline{f},\overline{g})(a,b)$ for all $b\in B$.
\end{lemma}
\begin{proof}
    We fix $(K,v)$, $B$, and $a$ as in the statement.
    We make a case distinction and prove that in each case, $\psi$ and $\pi$ exist, and it will be clear from the proof that they do not depend on $(K,v)$, $B$, or $a$.
    This is sufficient, as finitely many formulas can be coded into one, and finitely many definable maps can be combined in a tuple.
    Let $C:=Z_{\overline{f},\overline{g}}(K;B)$.
    We define
    \[\nu_0:=\max\{v(a-c):c\in C\}\text{ and }C_0:=\big\{c\in C:v(a-c)=\nu_0\big\}.\]
    
    We write $c_1\sim c_2$ whenever $c_1$ and $c_2$ are in $C_0$ and $v(c_1-c_2)>\nu_0$.
    Finally, we denote the map $\Res^K(f_\ell,g_\ell)(x;y_{\mathbf K})$ by $\rho_\ell(x;y_{\mathbf K})$.
    \newline

    {\underline{Case 1:}} There are $c_0$ and $c_1$ in $C_0$ with $c_0\not\sim c_1$.

    In this case, the image of $a$ under the map \[\pi(x;c_0,c_1):=\res_v\big(\frac{x-c_1}{c_0-c_1}\big)\] determines its image under $(\rho_\ell(x;b))_{\ell\leq N,b\in B}$.
    Indeed, suppose $a'\in K$ is such that $\pi(a';c_0,c_1)=\pi(a;c_0,c_1)$.
    By assumption, $v(c_1-c_0)=\nu_0=v(a-c_0)$, so we get $v(a'-a)>\nu_0$.
    Thus, by maximality of $\nu_0$, we get $v(a'-a)>v(c-a)$ for every $c\in C$.
    By the choice of $C$ and Lemma \ref{lemma: determining images under the residue map}, we conclude that $\rho_\ell(a';b)=\rho_\ell(a;b)$ for each $\ell\leq N$ and $b\in B$.\newline

    {\underline{Case 2:}} $|C_0/\sim|=1$.

    In this case, fix any $c_0\in C_0$ and find $c_1\in C\setminus C_0$ with 
    \[v(c_1-c_0)=\max\big\{v(c-c_0):c\in C\setminus C_0\big\}.\]
    Moreover, find $c_0,c_0'\in C_0$ such that $v(c_0-c_0')$ is minimal (that is, $c_0$ and $c_0'$ are as far apart as possible). 
    For each $b\in B$, $\ell\leq N$, and $c\in C_0$, let $m_c^f(\ell,b)$ be the multiplicity of $c$ as a root of $f_\ell(x;b)$, and similarly let $m_c^g(\ell,b)$ be the multiplicity of $c$ as a root of $g_\ell(x;b)$. Now, define
    \[k(\ell,b):=\sum_{c\in C_0}m_c^f(\ell,b)-m_c^g(\ell,b),\ S:=\big\{k(\ell,b):\  \ell\leq N,\  b\in B,\ v(f_\ell(a,b))=v(g_\ell(a,b))\big\}.\]
    Since $|k(\ell,b)|\leq \deg_x(f_\ell)+\deg_x(g_\ell)$ for each $\ell\leq N$ and $b\in B$, we get $|S|\leq m$ for some $m\in\mathbb N$ depending only on $\overline{f}$ and $\overline{g}$.
    Now, find $b_1,\dots,b_m$ in $B$ and $\ell_1,\dots,\ell_m\leq N$ such that $k(\ell_1,b_1),\dots,k(\ell_m,b_m)$ enumerates $S$ and such that $v(f_{\ell_i}(a,b_i))=v(g_{\ell_i}(a,b_i))$ for each $i\leq m$. 
    Write $\pi_1(x;y_{\mathbf K}^1,\dots,y_{\mathbf K}^m)$ for the definable map $\big(\rho_{\ell_1}(x;y_{\mathbf K}^1),\dots,\rho_{\ell_m}(x;y_{\mathbf K}^m)\big)$.

    Find, by Proposition \ref{prop: relative distality of valuation formulas} and Remark \ref{remark: facts about relative USHDs} (iv), a formula $\psi_0(x,s_{\mathbf K})$ as well as a definable map $\pi_0(x,s_{\mathbf K})$ into a finite power of $\mathbf{k}$ (depending only on $\overline{f}$ and $\overline{g}$) as well as a tuple $d$ from $B$, such that $\psi_0(x;d)\wedge\pi_0(x;d)=\pi_0(a;d)$ isolates the type of $a$ over $B$ with respect to each of the formulas
    \[\varphi_0^\ell(x;y_{\mathbf K}):\equiv\ v(f_\ell(x,y_{\mathbf K}))=v(g_\ell(x,y_{\mathbf{K}}))\neq \infty,\]
    with $\ell\leq N$.

    \underline{Claim:} Suppose $a'\in K$ is such that 
    \begin{enumerate}
        \item[(i)] $v(c_1-c_0)<v(a'-c_0)<v(c_0'-c_0)$,
        \item[(ii)] $\pi_1^K(a',b_1,\dots,b_m)=\pi_1^K(a,b_1,\dots,b_m)$,
        \item[(iii)] $\models \psi_0(a',d)$ and $\pi_0^K(a',d)=\pi_0^K(a,d)$. 
    \end{enumerate} 
    Then $\rho_\ell(a',b)=\rho_\ell(a,b)$ for all $b\in B$ and $\ell\leq N$.
    
    \begin{proof}[Proof of Claim]
        Fix $b\in B$ and $\ell\leq N$.
        Observe that if $\ell$ and $b$ are such that $v(f_\ell(a,b))\neq v(g_\ell(a,b))$ or $f(a,b)=g(a,b)=0$, then we get $\rho_\ell(a',b)=\rho_\ell(a,b)$ by (iii). 
        Thus, we may assume that $v(f_\ell(a,b))=v(g_\ell(a,b))$ and $g_\ell(a,b)\neq 0$.
        
        Note now that for $c\in C\setminus C_0$ we get
    \begin{align*}
        v(a'-a)=&v((a'-c_0)-(c_0-a))\\
        \geq &\min\big(v(a'-c_0),v(c_0-a)\big)\\
        \overset{by (i)}{>} &v(c_1-c_0)\\
        \geq &v(c-c_0)=v(c-a)
    \end{align*}
    using the ultrametric triangle inequality and maximality of $v(c_1-c_0)<\nu_0$. It follows that
    \[v\Big(\frac{a'-c}{a-c}-1\Big)>0.\]
    Thus,
    \begin{align}\label{eq: res lemma eq 1}
        \res_v\Big(\frac{a-c}{a'-c}\Big)=1=\res_v\Big(\frac{a'-c}{a-c}\Big),\ \text{ for all }c\in C\setminus C_0.
    \end{align}
    
    On the other hand, note that 
    \begin{equation}\label{eq: res lemma eq 3}
        \res_v\Big(\frac{a'-c}{a'-c_0}\Big)=\res_v\Big(\frac{a-c}{a-c_0}\Big)=1, \ \text{ for all }c\in C_0.
    \end{equation}
    Indeed, given $c\in C_0$, we have $v(c_0-c)>\nu_0=v(a-c)$ by the assumption of Case 2.
    It follows that 
    \[v\Big(\frac{a-c}{a-c_0}-1\Big)>0,\]
    which yields the second equality in (\ref{eq: res lemma eq 3}).
    Similarly, we have $v(c-c_0)>v(a'-c_0)$ by assumption $(i)$ and choice of $c_0$ and $c_0'$, which yields the first equality of (\ref{eq: res lemma eq 3}).
    
    We conclude that
    \begin{align*}
        \rho_\ell(a,b)&=\res_v\Big(\frac{f_\ell(a',b)}{g_\ell(a',b)}\cdot\prod_{c\in C}\big(\frac{a-c}{a'-c}\big)^{m_c^f(\ell,b)-m_c^g(\ell,b)}\Big)\\
        &= \rho_\ell(a',b)\cdot \res_v\Big(\prod_{c\in C}\big(\frac{a-c}{a'-c}\big)^{m_c^f(\ell,b)-m_c^g(\ell,b)}\Big)\\
        &\overset{\text{by }(\ref{eq: res lemma eq 1})}{=}\rho_\ell(a',b)\cdot \res_v\Big(\prod_{c\in C_0}\big(\frac{a-c}{a'-c}\big)^{m_c^f(\ell,b)-m_c^g(\ell,b)}\Big)\\
        &\overset{\text{by }(\ref{eq: res lemma eq 3})}{=}\rho_\ell(a',b)\cdot \res_v\Big(\frac{a-c_0}{a'-c_0}\Big)^{k(\ell,b)}
    \end{align*}
    where we used the fact that $0=v(\frac{f_\ell(a,b)}{g_\ell(a,b)})=v(\frac{f_\ell(a',b)}{g_\ell(a',b)})$ in the second step.
    
    If $k(\ell,b)=0$, we get $\rho_\ell(a',b)=\rho_\ell(a,b)$ immediately.
    Otherwise, we can find $i\leq m$ such that $k(\ell_i,b_i)=k(\ell,b)$. 
    By assumption (ii) of the claim, and (\ref{eq: res lemma eq 3}), we obtain
    \[1=\frac{\rho_{\ell_i}(a,b_i)}{\rho_{\ell_i}(a',b_i)}=\res_v\Big(\frac{a-c_0}{a'-c_0}\Big)^{k(\ell_i,b_i)}=\res_v\Big(\frac{a-c_0}{a'-c_0}\Big)^{k(\ell,b)},\]
    and conclude that $\rho_\ell(a',b)=\rho_\ell(a,b)$.
    \end{proof}
\end{proof}

We can now easily derive relative distality of formulas of the form \[\varphi_{\overline{f},\overline{g},P}(x;y_{\mathbf K},y_{\mathbf k})\equiv P\Big(\Res(\overline{f},\overline{g})(x,y_{\mathbf K}),\ y_{\mathbf k}\Big)=0,\]
where $x$ is unary in the field-sort $\mathbf K$.

\begin{proof}[Proof of Proposition \ref{prop: res formulas are relatively distal}]
    Let $(K,v)\models\ACVF$ and $A\fin K^{|y_{\mathbf K}|}\times k^{|y_{\mathbf k}|}$.
    We let $B$ be the projection of $A$ to the $y_{\mathbf K}$-part.
    Note that if $a$ and $a'$ are such that $\Res_v(\overline{f},\overline{g})(a,b)=\Res_v(\overline{f},\overline{g})(a',b)$ for all $b\in B$, then we trivially get $\tp_{\varphi_{\overline{f},\overline{g},P}}(a/A)=\tp_{\varphi_{\overline{f},\overline{g},P}}(a'/A)$.
    By Lemma \ref{lemma: determining images in the residue field} and quantifier elimination (Fact \ref{fact: QE in ACVF}), the assumptions of Lemma \ref{lemma: algebraic parameters in USHD} are satisfied for $\varphi_{\overline{f},\overline{g},P}(x;y_{\mathbf K},y_{\mathbf k})$.
    We conclude that $\varphi_{\overline{f},\overline{g},P}(x;y_{\mathbf K},y_{\mathbf k})$ has a $\mathbf k$-still in $\ACVF$.
\end{proof}

Finally, we have all the ingredients to prove that $\ACVF$ is distal relative to $\mathbf k$.

\begin{proof}[Proof of Theorem \ref{thm: ACVF distal relative to k}]
    By Fact \ref{fact: QE in ACVF}, any atomic $L_{\Gamma ,k}$-formula $\varphi(x_{\mathbf{K}},y)$, where $x_{\mathbf K}$ is unary, is either of the form $\varphi_{f,g,\sigma,\tau}$ or $\varphi_{\overline{f},\overline{g},P}$ defined in \eqref{formula form 2} and $\eqref{formula form 3}$. 
    Thus, by Lemma \ref{lemma: relative distality of k and Gamma formulas}, Proposition \ref{prop: relative distality of valuation formulas}, and Proposition \ref{prop: res formulas are relatively distal}, any atomic $L_{\Gamma ,k}$-formula $\varphi(x,y)$ with unary $x$ has a $\mathbf k$-still in $\ACVF$.
    By Lemma \ref{lem: Boolean combinations} and Lemma \ref{lem: relative distality in one variable enough}, it follows that $\ACVF$ is distal relative to $\mathbf{k}$.
\end{proof}

\section{Incidence bounds for fields with finite residue field}\label{sec: applications}

Fix throughout this section a field $F$ with a valuation $v$ on $F$ such that its residue field is finite, say $\mathbf{k}(F)=\mathbb F_q$, where $q$ is a prime power. 
For concreteness, the reader may assume that $F$ is the function field $\mathbb F_q(t)$ with the \emph{t-adic} valuation 
\[v\Bigg(\frac{\sum_{i=0}^k\lambda_it^i}{\sum_{j=0}^\ell\mu_jt^j}\Bigg):=\min\{i:\lambda_i\neq 0\}-\min\{j:\mu_j\neq 0\}.\]

Other examples of fields admitting a valuation with finite residue field include the field $\mathbb F_q((t))$ of Laurent series, function fields of varieties over $\mathbb F_q$, the p-adics, and all finitely generated extensions of finite fields (see \cite[Section 4]{bays_martin}).

We work throughout this section in the $L_{\Gamma ,k}$-theory $\ACVF$.
All formulas will be $L_{\Gamma ,k}$-formulas, and we will refer to a $\mathbf k$-still in $\ACVF$ simply as a $\mathbf{k}$-still (for a given formula). 
All variables will be of the field sort, unless they carry a subscript $\mathbf{k}$, in which case they are variables of the residue field sort.

\subsection{General bound}

Say that a subset $E$ of $F^x\times F^y$ is \emph{valuative} if $E$ is a Boolean combination of sets of the form
\[W_{f,g}:=\big\{(a,b)\in F^x\times F^y:\ v(f(a,b))<v(g(a,b))\big\}\] 
where $f$ and $g$ range over the polynomial ring $F[x,y]$.
We say that $E$ has \emph{description complexity} at most $k$ if $\max(|x|,|y|)\leq k$ and $E$ is a Boolean combination of at most $k$-many sets $W_{f,g}$ where $f$ and $g$ have degree at most $k$ as polynomials in the multi-variable $xy$.

Note that in particular, if $E$ is \emph{algebraic} -- that is, $E$ is a Boolean combination of sets of the form 
\[Z_f=\{(a,b)\in F^x\times F^y:f(a,b)=0\},\]
then it is valuative.
Indeed, in any valued field $F$, the set $Z_f$ coincides with the complement of $W_{f,0}$.
Thus, Theorem \ref{thm: intro 1} is indeed a special case of the below theorem.

\begin{theorem}\label{thm: incidence bounds valued fields} 
    Let $k$ and $d$ be positive integers. There is some $\epsilon>0$ such that for any valuative set $E\subseteq F^x\times F^y$ of description complexity at most $k$, and any $A\fin F^x$, $B\fin F^y$, and $s\in\mathbb N$ with the property that $E(A,B)$ is $K_{d,s}$-free, we have
    \[\big|E(A,B)\big|=O_{d,s,k,q}\big(|A|^{1-\epsilon}|B|^{1-\frac{1}{d}+\frac{\epsilon}{d}}+|A|+|B|\big).\]
\end{theorem}
\begin{proof}    
    For a multi-variable $z$ of length $m$ and $\overline{\ell}\in\{0,\dots,k\}^m$, we write $z^{\overline{\ell}}$ for the term
    $z_1^{\ell_1}\cdot\ldots \cdot z_m^{\ell_m}$.
    Let $I_k$ denote the set of tuples $\overline{\ell}\in\{0,\dots,k\}^{2k}$ with $\sum_i\ell_i\leq k$.
    
    Since $E$ is of description complexity $\leq k$, we may assume that $|x|=|y|=k$, and $E=\varphi^F$ for an $L_{\Gamma ,k}(F)$-formula $\varphi(x,y)$, which is a Boolean combination of $L_{\Gamma ,k}(F)$-formulas of the form 
    \[\chi_{\lambda,\mu}(x,y):\equiv v\Big(\sum_{\overline{\ell}\in I_k} \lambda_{\overline{\ell}}\cdot(xy)^{\ell}\Big)<v\Big(\sum_{\overline{\ell}\in I_k} \mu_{\overline{\ell}}\cdot(xy)^{\ell}\Big),\]
    where $\lambda,\mu\in F^{I_k}$.
    We find $(K,v')\models \ACVF$ extending $(F,v)$ as a valued field, so in particular $\res_{v'}(F)=\res_v(F)=\mathbb F_q$.
    Such an extension exists; see for example \cite[Theorem 3.1.2]{MR2183496}.
    
    By Theorem \ref{thm: ACVF distal relative to k} and Lemma \ref{lem: relative USHDs with parameters}, each formula $\chi_{\lambda,\mu}(x,y)$ has a $\mathbf{k}$-still in $(K,v')$, and uniformly in the sense that there is an $L_{\Gamma ,k}$-formula $\psi(x,z,u_{\mathbf k};w)$, such that $\psi_{\lambda,\mu}(x;z;u_{\mathbf k}):=\psi(x,z,u_{\mathbf k};\lambda\mu)$ is a $\mathbf{k}$-still for $\chi_{\lambda,\mu}(x,y)$ in $(K,v')$, for any $\lambda,\mu\in K^{I_k}$.
    Let $(\lambda_i,\mu_i)_{i\leq k}$ enumerate the set of pairs $(\lambda,\mu)\in F^{I_k}\times F^{I_k}$ where $\chi_{\lambda,\mu}$ appears in the formula $\varphi(x;y)$.
    Let $\overline{z}:=(z_{i})_{i\leq k}$ be a concatenation of $k$-many pairwise-disjoint copies of $z$, and let $\overline{u}_{\mathbf k}:=((u_{\mathbf k})_i)_{i\leq k}$ be a concatenation of $k$-many disjoint copies of $u_{\mathbf k}$.
    Now, let $\psi(x;\overline{z};\overline{u}_{\mathbf k})$ be the $L_{\Gamma,k}(F)$-formula 
    \[\bigwedge_{i\leq k}\psi_{\lambda_i,\mu_i}(x,z_i,(u_{\mathbf k})_i).\] 
    Then $\psi$ is a $\mathbf k$-still for $\varphi$ in $(K,v')$.
    Since $\dcl(F)\cap \mathbf{k}(K)=\mathbb F_q$ has cardinality $q$ and since
    \[E=\varphi^F=\varphi^K(F^x,F^y),\]
    the result follows from Corollary \ref{cor: k-stills give zarankiewicz bounds} with $\epsilon=\frac{d-1}{td-1}$, where $t=\frac{|\overline{z}|}{|y|}$ is the arity of $\psi$. 
\end{proof}

\begin{remark}
    \begin{enumerate}
        \item[(i)] Under the mild symmetry assumption that $|B|^{\frac{1}{d}+\delta}\leq |A|\leq |B|^{\frac{d-1}{\epsilon}-d\delta}$ for some positive $\delta$, Theorem \ref{thm: incidence bounds valued fields} implies the bound
        \[|E(A,B)|=O\big(|A|^{1-\frac{\epsilon}{d}}|B|^{1-\frac{1}{d}-\epsilon\delta}+|B|\big),\]
        providing a power-saving when compared with the classical bound of $O(|A||B|^{1-\frac{1}{d}}+|B|)$ coming from Fact \ref{fact: Kövari-Sos-Turan}.

        \item[(ii)] By Remark \ref{remark: poly dependence relative distality bound}, the constant factor in the bound from Theorem \ref{thm: incidence bounds valued fields} is polynomial in $q$.
        \item[(iii)] Note that a relation $E\subseteq F^x\times F^y$ is valuative if and only if it is defined by a quantifier-free $L_{\Gamma ,k}(F)$-formula $\varphi(x;y)$.
        Recall that we write $\varphi^F$ for the interpretation of $\varphi$ in $(F,v)$.
        If $(K,v')\geq (F,v)$ is a valued field extension, then $\varphi^F=E$ coincides with the restriction $\varphi^K(F^x,F^y)$. 
    \end{enumerate}
\end{remark}

\subsection{Point-line incidence bound}\label{sec: point-line}

We show in this section that in the case where $E\subseteq F^2\times F^2$ is point-line incidence, Theorem \ref{thm: incidence bounds valued fields} holds with $\epsilon=\frac{1}{9}$.
That is, we prove the following.

\begin{theorem}\label{thm: point line explicit bound}
    Let $(F,v)$ be a valued field with residue field $\mathbb F_q$.
    Let $P$ and $L$ be finite sets of points and lines in $F^2$, respectively, and let $I$ be the incidence relation.  
    Then 
    \[\big|I(P,L)\big|=O_q\big(|P|^{\frac{8}{9}}|L|^{\frac{5}{9}}+|P|+|L|\big).\]
    In particular, if $n=|P|=|L|$, then 
    \[|I(P,L)|=O_q(n^{\frac{3}{2}-\frac{1}{18}}).\]
\end{theorem}

In the above theorem, we may ignore vertical lines, that is, lines of the form $\{(a_1,a_2)\in F^2: a_1=c\}$ for some fixed value $c\in F$. 
Indeed, each point is incident with exactly one vertical line, so such incidences contribute at most $|P|$ to $|I(P,L)|$.
Thus, in what follows, we identify lines with pairs $(b_1,b_2)\in F^2$, and $I\subseteq F^2\times F^2$ is precisely the interpretation of the formula
$x_2=x_1y_1+y_2$
in the field $F$.

Recall the definition of a piecewise $\mathbf k$-still from the paragraph below Remark \ref{remark: facts about relative USHDs}. 

The strategy for proving Theorem \ref{thm: point line explicit bound} is to construct a piecewise $\mathbf k$-still for the formula
\[\varphi_I(x;y):\equiv x_2=x_1y_1+y_2,\]
which induces a DCD of suitably low exponent for $(\varphi_I)^F=I$.
In the first step, we construct a piecewise $\mathbf{k}$-still $(\psi_i(x_2,z,u_{\mathbf k}))_{i}$ for the formula $\varphi_I^0(x_2;x_1,y):\equiv \varphi_I(x,y)$.
Following the argument of Lemma \ref{lem: relative distality in one variable enough}, we consider, for each $i$, the formula $\theta_i(x_1;z,u_{\mathbf{k}},y)$ where $\theta_i(a_1;c,\alpha,b)$ expresses that the set defined by $\psi_i(x_2;c,\alpha)$ is not crossed by the set defined by $\varphi_I^0(x_2;a_1,b)$.
In other words, the set defined by $\psi_i(x_2;c,\alpha)$ is either the singleton $\{a_1b_1+b_2\}$ or a set not containing $a_1b_1+b_2$.
In the second step, we construct a piecewise $\mathbf{k}$-still for each such formula $\theta_i(x_1;z,u_{\mathbf k},y)$.
This is covered by Lemma \ref{lem: linear swiss cheese} and Lemma \ref{lemma: linear residue images}.
Putting the two steps together, we obtain a piecewise $\mathbf{k}$-still (of arity seven) for $\varphi_I(x,y)$ in Proposition \ref{proposition: DCD for point line}. 
Finally, we show that the exponent of the resulting DCD for $I$ can be reduced to five, rather than the trivial upper bound seven coming from the arity, by removing some cells which are not needed.
The final step is essentially already contained in Lemma \ref{lem: linear swiss cheese} and Lemma \ref{lemma: linear residue images}.

\begin{remark}\label{remark: suboptimal bounds}
    The exponents in the bound given by Theorem \ref{thm: point line explicit bound} are almost certainly not optimal, and in fact can likely be improved with the same methods, by carefully removing unnecessary cells from the DCD given by Proposition \ref{proposition: DCD for point line}.
\end{remark}

We begin by constructing a $\mathbf{k}$-still for the formula $\varphi_I^0(x_2;x_1,y)$.
For notational convenience, let $z$ be a ternary variable, and write $\tau(z)$ for the term $z_1z_2+z_3$.
Thus, $\varphi_I^0(x_2;z)$ is the formula $x_2=\tau(z)$.
Let $(z^i)_i$ be a family of disjoint copies of $z$.
Finally, let $u_{\mathbf k} $ denote a unary variable of sort $\mathbf k$.

\begin{lemma}\label{lemma: step 1 of point line incidence bound}
    The formulas 
    \begin{align*}
        &\psi_1(x_2;z^1,z^2):\equiv v(x_2-\tau(z^1))<v(\tau(z^2)-\tau(z^1)),\\
        &\psi_2(x_2;z^1,z^2,z^3):\equiv v(\tau(z^2)-\tau(z^1))<v(x_2-\tau(z^1))<v(\tau(z^3)-\tau(z^1)),\\
        &\psi_3(x_2;z^1,z^2;u_{\mathbf k} ):\equiv \Res\big(x_2-\tau(z^1),\tau(z^2)-\tau(z^1)\big)=u_{\mathbf k},\text{ and} \\
        &\psi_4(x_2;z^1):\equiv x_2=\tau(z^1) 
    \end{align*}
    form a piecewise $\mathbf{k}$-still for $\varphi_I^0(x_2;z)$, witnessed by $f_3(x_2;z^1,z^2):=\Res\big(x_2-\tau(z^1),\tau(z^2)-\tau(z^1)\big)$.
\end{lemma}
\begin{proof}
    Given $(K,v)\models \ACVF$ and $B\fin K^3$, each point $\tau(b)$ with $b\in B$ can be identified with the closed ball of radius $\infty$ around $\tau(b)$. 
    Indeed, $b^f(\tau(b),\infty)$ is the singleton $\{\tau(b)\}$.
    Let $\mathcal{B}(B)$ be the closure of the collection of these balls under joins; $\mathcal{B}(B)$ contains all balls of the form $b^f\big(\tau(b),\tau(b')-\tau(b)\big)$ with $b,b'\in B$.
    Clearly, the $\varphi_I^0$-type of any point $a\in K$ over $B$ is determined by the containment-type of $a$ over $\mathcal{B}(B)$.
    A careful analysis of the case distinction in Lemma \ref{lemma: containment has USHD} in this special case (where the collection of balls consists only of points and their joins) now proves the statement.
    More precisely, $\psi_1$ covers the case where $a\notin\bigcup\mathcal{B}(B)$ (that is Case 1 in Lemma \ref{lemma: containment has USHD}), $\psi_2$ covers the case where $a$ is contained in some $b^o(\tau(b),\tau(b')-\tau(b))$ and not in any subballs of it contained in $\mathcal{B}(B)$ (that is Case 3.1 in the Lemma), and $\psi_3$  together with $f_3$ covers the case where $a$ is contained in some $b^f(\tau(b),\tau(b')-\tau(b))\in\mathcal{B}(B)$, but it is not contained in $b^o(\tau(b^*),\tau(b')-\tau(b))$ for any $b^*\in B$ (that is Case 3.2 in the Lemma).
    Finally, $\psi_4$ covers the case where $a=\tau(b)$ for some $b\in B$.
\end{proof}

We move on to proving the necessary ingredients for step $2$ as described above.
Note that we prove slightly more general statements than is necessary for the proof of Theorem \ref{thm: point line explicit bound}.
Call an $L_{\Gamma ,k}$-formula $\varphi(x,y)$ with $|x|=1$ a \emph{basic valuative formula} if it is of the form $v(f(x,y))<v(g(x,y))$ where $f,g \in \mathbb Z[x,y]$.
If $f(x,y)$ and $g(x,y)$ are linear in $x$, then we say that $\varphi(x,y)$ is \emph{linear}.
In what follows, we will refer to expressions of the form $\frac{\tau}{\tau'}(y)$ as \emph{terms}, where $\tau$ and $\tau'$ are $L_{\text{ring}}$-terms.
Whenever $\psi(x)$ is an $L_{\Gamma ,k}$-formula, we will write $\psi(\frac{\tau}{\tau'}(y))$ as shorthand for the formula 
\[\exists z\big(z\cdot\tau'(y)=1\wedge \psi(z\cdot \tau(y))\big).\]

\begin{lemma}\label{lem: linear swiss cheese}
    Let $\varphi(x,y)$ be a Boolean combination of linear basic valuative formulas. Then $\varphi$ has a piecewise $\mathbf k$-still of arity three.
    Moreover, this induces a DCD of exponent two, with a constant factor $C(q)$ polynomial in $q$, for the relation $\varphi^F\subseteq F^x\times F^y$.
\end{lemma}
\begin{proof}
    For a tuple $\tau=(\tau_1,\tau_2,\tau_3,\tau_4)$ of $L_{\text{ring}}$-terms, consider the formula 
    \begin{align}\label{eq: linear formula}
        \psi_\tau(x,y)\ :\equiv\  v(\tau_1(y)x+\tau_2(y))<v(\tau_3(y)x+\tau_4(y)).
    \end{align}
    For $b\in K^y$, let $\mathcal{B}_\tau(b)$ be the collection of balls which have one of the following centers (excluding any undefined expressions)
    \begin{equation}\label{eq: linear exponent center terms}
        -\tau_2(b),\ -\tau_4(b),\ -\frac{\tau_2}{\tau_1}(b),\ -\frac{\tau_4}{\tau_3}(b)
    \end{equation}
    and whose radius is the image under $v$ of one of the following expressions (excluding any undefined expressions)
    \begin{align}\label{eq: linear exponent radius terms}
        0,\ \frac{\tau_4}{\tau_1}(b),\ \frac{\tau_2}{\tau_3}(b),\ \Big(\frac{\tau_4}{\tau_3}-\frac{\tau_2}{\tau_1}\Big)(b),\ \frac{\tau_3}{\tau_1}\Big(\frac{\tau_4}{\tau_3}-\frac{\tau_2}{\tau_1}\Big)(b),\ \frac{\tau_1}{\tau_3}\big(\frac{\tau_4}{\tau_3}-\frac{\tau_2}{\tau_1}\big)^{-1}(b).
    \end{align}

    \underline{Claim:} The set $\psi_\tau^K(K;b)$ is a Boolean combination of balls in $\mathcal{B}_\tau(b)$.
    
    \begin{proof}[Proof of Claim.]
    If $\tau_1(y)=\tau_3(y)=0$, the formula $\psi_\tau$ trivializes.
    If $\tau_1(y)=0\neq\tau_3(y)$, the formula $\psi_\tau$ is equivalent to $x\in b^{o}(-\frac{\tau_4}{\tau_3}(y),\frac{\tau_2}{\tau_3})$, and the case $\tau_1(y)\neq 0=\tau_3(y)$ is similar.
    Thus, we assume that $\tau_1(y)\neq0\neq\tau_3(y)$. In this case, the formula $\psi_\tau$ is equivalent to
    \begin{align*}
        v\Big(x+\frac{\tau_2}{\tau_1}(y)\Big)-v\Big(x+\frac{\tau_4}{\tau_3}(y)\Big)<v\Big(\frac{\tau_3}{\tau_1}(y)\Big)
    \end{align*}
    and the statement follows from the claim proved in Lemma \ref{lemma: swiss cheeses}.
    \end{proof}

    By assumption, $\varphi(x,y)$ is a Boolean combination of formulas of the form $\psi_\tau$ as in (\ref{eq: linear formula}).
    Let $T(\varphi)$ be the finite set of tuples $\tau$ where $\psi_\tau$ appears in $\varphi$.
    Fix a set $B\fin K^y$ and
    let $\mathcal{B}(B)$ be the union of all sets $\mathcal{B}_\tau(b)$ with $b\in B$ and with $\tau\in T(\varphi)$.
    By the claim, for any $a\in K$, the $\varphi$-type of $a$ over $B$ is isolated by the containment-type of $a$ over $\mathcal{B}(B)$.
    Containment in balls has a piecewise $\mathbf{k}$-still of arity three by Lemma \ref{lemma: containment has USHD} and Remark \ref{remark: about containment and k-stills} (ii).
    Existence of a piecewise $\mathbf k$-still of arity three for $\varphi$ now follows from the fact that each ball in $\mathcal{B}(B)$ is of the form $b^\star(\sigma(b),r(b))$ for some $b\in B$, $\star\in\{o,f\}$, and with $\sigma$ and $r$ coming from the finite set of terms defined in (\ref{eq: linear exponent center terms}) and (\ref{eq: linear exponent radius terms}), with $\tau\in T(\varphi)$.
    A piece will correspond to one of the pieces of the piecewise $\mathbf k$-still for containment from Lemma \ref{lemma: containment has USHD} and a pair $(\sigma,r)$.
    The moreover-part of the statement then follows from Remark \ref{remark: about containment and k-stills} (ii) and the fact that $|\dcl(F)\cap\mathbf{k}(K)|=q$.
\end{proof}

The second step of our proof strategy requires us to find $\mathbf{k}$-stills for formulas including a variable in the residue field sort.
We call a partitioned formula $\varphi(x;y,u_\mathbf{k})$ with $|x|=1$ a \emph{basic residue formula} if it is of the form
\[\Res(f(x,y),g(x,y))={u}_{\mathbf k}\]
where $f,g\in\mathbb Z[x,y]$.
We say that such a formula $\varphi$ is \emph{linear} if $f$ and $g$ are linear in $x$.

\begin{lemma}\label{lemma: linear residue images}
    Let $\varphi(x;y,u_{\mathbf k})$ be a linear basic residue formula. Then $\varphi$
    has a piecewise $\mathbf{k}$-still of arity five. 
    Moreover, this induces a DCD of exponent three for $\varphi^F$ with a constant factor $C(q)$ polynomial in $q$.
\end{lemma}
\begin{proof}
    We may assume that $y=(y_1,y_2)$ is binary and $\varphi$ is of the form 
    \[\Res(\sigma(y)(x-y_1),\tau(y)(x-y_2))=u_{\mathbf k}\]
    where $\sigma$ and $\tau$ are $L_{\text{ring}}$-terms.
    The proof of the general case is analogous.
    If $\tau(y)$ is the constant zero term, then $\varphi$ trivializes, so we may ignore the term $\tau$ by replacing $\sigma$ with $\frac{\sigma}{\tau}$.

    Fix $(K,v)\models \ACVF$, $a\in K$ and finite set $B\subseteq K^2$. 
    Write $f(x,y):=\Res\big(\sigma(y)(x-y_1),(x-y_2)\big)$.
    We follow the proof of Lemma \ref{lemma: determining images in the residue field}; 
    we make a case distinction and show that in each case, there is a formula $\psi(x,\overline{b})$ and possibly a definable map $\rho(x,\overline{b})$ into $\mathbf k$, with $\psi$ and $\rho$ independent of $a$ and $B$, and with $\overline{b}\in B^{\leq 5}$, such that the formula $\psi(x;\overline{b})\wedge\rho(x;\overline{b})=\rho(a,\overline{b})$ determines the image of $a$ under the family of maps $(f^K(x;b))_{b\in B}$.
    This induces a piecewise $\mathbf k$-still for $\varphi$ by setting $\psi'(x;z;u_{\mathbf k}):\equiv \psi(x;z)\wedge \rho(x,z)=u_{\mathbf k}$.

    Let $C$ denote the set of points $c\in K$ for which $(c,d)\in B$ or $(d,c)\in B$ for some $d\in K$.
    Let $\nu_0:=\max\{v(a-c):c\in C\}$ and $C_0:=\{c\in C:v(a-c)=\nu_0\}$.
    For $c,c'\in C_0$, we write $c\sim c'$ whenever $v(c-c')>\nu_0$.
    Thus, $c\not\sim c'$ if and only if $v(c-c')=\nu_0$.
    \\

    \underline{Case 1:} There exist $c^1,c^2\in C_0$ with $c^1\not\sim c^2$. 
    
    Set $\alpha:=\res_v(\frac{a-c^1}{c^2-c^1})$.
    By Lemma \ref{lemma: determining images under the residue map}, the image of $a$ under $(f^K(x;b))_{b\in B}$ is determined by the formula
    \begin{align}\label{eq: linear residue, piece 1}
        \Res(x-c^1,c^2-c^1)=\alpha.
    \end{align}

    \underline{Case 2:}
    $|C_0/\sim|=1$.
    
    Fix some $c^1\in C_0$ and find $c^2\in C\setminus C_0$ such that $v(c^2-c^1)=\max\{v(c-c^1)\}_{c\in C\setminus C_0}$. 
    On the other hand, find $c^3\in C_0$ such that $v(c^3-c^1)=\min\{v(c-c^1)\}_{c\in C_0}$.
    Define 
    \[\psi(x;c^1,c^2,c^3):\equiv\ v(c^2-c^1)<v(x-c^1)<v(c^3-c^1).\]

    By the second part of the formula, we get
    \begin{align}\label{eq: residue linear distance from C_0}
        a'\in \psi^K(K;c^1,c^2,c^3)\implies \bigwedge_{c\in C_0}\res_v\big(\frac{a'-c}{a'-c^1}\big)=1.
    \end{align}
    By the first part of the formula, we get
    \begin{align}\label{eq: residue linear far aways balls}
        a'\in \psi^K(K;c^1,c^2,c^3)\implies  \bigwedge_{c\in C\setminus C_0}\res_v\big(\frac{a'-c}{a-c}\Big)=1. 
    \end{align}
    In particular, for $S:=\{(b_1,b_2)\in B: b_1\in C_0\not\leftrightarrow b_2\in C_0\}$ and $S':=B\setminus S$, we get
    \begin{align}\label{eq: linear residue S'}
        a'\in\psi^K(K;c^1,c^2,c^3)\implies \bigwedge_{b\in S'} f^K(a'
        ,b)=f^K(a,b).
    \end{align}

    \underline{Case 2.1:} There is $b^1\in S$ with $f^K(a,b^1)\neq0$.

    In this case, setting $\alpha:=f(a,b^1)$, an argument as in Case 2 of the proof of Lemma \ref{lemma: determining images in the residue field} shows that  the image of $a$ under $(f^K(x;b))_{b\in B}$ is determined by the formula
    \begin{align}\label{eq: linear residue, piece 2}
        \psi(x;c^1,c^2,c^3)\wedge f(x;b^1)=\alpha.
    \end{align}

    \underline{Case 2.2:} For each $b\in S$, we have $f^K(a,b)=0$. 

    Equivalently, for each $b=(b_1,b_2)\in S$, we have 
    \[v\big(\sigma(b)\big)+v\big(a-b_1)\neq v(a-b_2).\]
    We define $S^1:=\{b\in S:b_1\in C_0\}$, $S^2:=\{b\in S:b_2\in C_0\}$, and 
    \[\mathcal{V}(a):=\ \big\{\overbrace{v(a-b_2)-v(\sigma(b))}^{=:\mu_1(a,b)}: b\in S^1\big\}\ \cup\ \big\{\overbrace{v(\sigma(b))+v(a-b_1)}^{=:\mu_2(a,b)}: b\in S^2\big\}.\]
    Recall that $\nu_0=v(a-c)$ for each $c\in C_0$ and find consecutive points $\mu_{i_1}(a,b^1),\mu_{i_2}(a,b^2)\in\mathcal{V}(a)$ with $\mu_{i_1}(a,b^1)<\nu_0<\mu_{i_2}(a,b^2)$.
    Then the image of $a$ under $(f^K(x;b))_{b\in B}$ is determined by the formula
    \begin{equation}\label{equation: residue images ordering}
        \psi(x;c^1,c^2,c^3)\wedge \mu_{i_1}(x;b^1)<v(x-c^1)<\mu_{i_2}(x;b^2).
    \end{equation}
    Indeed, if $a'\in \psi(K;c^1,c^2,c^3)$, then $\mu_i(a',b)=\mu_i(a,b)$ for each $b\in S^i$ by (\ref{eq: residue linear far aways balls}).
    We conclude that $\mu_{i_1}(a',b^1)$ and $\mu_{i_2}(a',b^2)$ are consecutive points in $\mathcal{V}(a')$ and thus by (\ref{equation: residue images ordering}),
    \begin{align*}
        &\forall b\in S^1:\ v(a'-b_2)-v(\sigma(b))\neq v(a'-c^1)\overset{\text{by }(\ref{eq: residue linear distance from C_0})}{=}v(a'-b_1),\\
        &\forall b\in S^2:\ v(\sigma(b))+v(a'-b_1)\neq v(a'-c^1)\overset{\text{by }(\ref{eq: residue linear distance from C_0})}{=} v(a'-b_2).
    \end{align*}
    In particular, $f^K(a',b)=0$ for each $b\in S$, and so $(f^K(a',b))_{b\in B}=(f^K(a,b))_{b\in B}$,
    by (\ref{eq: linear residue S'}).\newline

    This concludes the case distinction. 
    For the ``Moreover"-part of the statement, restrict attention to the sub-field $F\leq K$, where $\mathbf k(F)=\mathbb F_q$.
    The formulas defined in the case distinction above yield a family of relations, as follows.
    Formula (\ref{eq: linear residue, piece 1}) from Case 1 induces a family of relations, parametrised by $\alpha\in \mathbb F_q$, and a tuple $\epsilon\in\{1,2\}^2$; namely
    \[R_{1}^{\epsilon,\alpha}:=\Big\{(a,b^1,b^2)\in F\times F^2\times F^2: \res_v\Big(\frac{a-b^1_{\epsilon(1)}}{b^2_{\epsilon(2)}-b^1_{\epsilon(1)}}\Big)=\alpha\Big\}.\]
    The tuple $\epsilon$ encodes whether $c_1$ and $c_2$, respectively, are the first or second coordinate of a point in $B$.
    Similarly, the formula (\ref{eq: linear residue, piece 2}) from Case 2.1 induces a family of relations

    \[R_{2.1}^{\epsilon,\alpha}:=\Big\{(a,b^1,b^2,b^3,b^4)\in F\times (F^2)^4:\ a\in\psi^K(K;b^1_{\epsilon_1},b^2_{\epsilon_2},b^3_{\epsilon_3})\text{ and } f^{K}(a;b^4)=\alpha\Big\}\]
    indexed by $\epsilon\in\{1,3\}^3$ and $\alpha\in\mathbb F_q$, and formula (\ref{equation: residue images ordering}) from Case 2.2 induces a family of relations 

    \begin{align*}
        R_{2.2}^{\epsilon,i}:=\Big\{(a,b^1,b^2,b^3,b^4,b^5)\in F\times (F^2)^5:\ &a\in\psi^K(K;b^1_{\epsilon_1},b^2_{\epsilon_2},b^3_{\epsilon_3})\text{ and } \\
        & \mu_{i_1}(a,b^4)<v(a-b^1_{\epsilon_1})<\mu_{i_2}(a,b^5)\Big\},
    \end{align*}
    indexed by $i\in\{1,2\}^2$ and $\epsilon\in\{1,2\}^3$.
    Combining these yields a DCD $\mathcal{T}$ for the binary relation $\varphi^F\subseteq F\times (F^2\times \mathbb F_q)$, where for each finite set $D\subseteq F^2\times \mathbb F_q$, the cells in $\mathcal{T}(D)$ are fibers of the form
    \[\big(R^{\epsilon,\alpha}_1\big)_{(b^1,b^2)},\ \big(R^{\epsilon,\alpha}_{2,1}\big)_{(b^1,b^2,b^3,b^4)},\text{ or }\big(R^{\epsilon,i}_{2.2}\big)_{(b^1,b^2,b^3,b^4,b^5)},\]
    where each $b^j$ is in the projection of $D$ to the first two coordinates, written $B:=\pi_{F^2}(D)$.
    By simply counting parameters, this DCD can be seen to have exponent $\leq 5$.
    To reduce it to three, we count more carefully, and remove unnecessary cells.  
    Firstly, a cell of the form $\big(R^{\epsilon,\alpha}_{2.1}\big)_{(b^1,b^2,b^3,b^4)}$ should only be included in $\mathcal{T}(D)$ if 
    \begin{align}\label{eq: linear residue b2 determined}
        v(b^2_{\epsilon_2}-b^1_{\epsilon_1})=\max\big\{v(b_{\epsilon_2}-b^1_{\epsilon_1}):b\in B\text{ and } v(b_{\epsilon_2}-b^1_{\epsilon_1})<v(b^3_{\epsilon_3}-b^1_{\epsilon_1})\big\}.
    \end{align}
    After fixing  $b^1$, $b^3$, $b^4$, $\epsilon$ and $\alpha$, this leaves at most one cell of the form $(R_{2.1}^{\epsilon,\alpha})_{(b^1,b^*,b^3,b^4)}$, by the definition of the formula $\psi$. 
    In particular, after removing cells accordingly, there are at most $O(q|B|^3)$-many cells of the form $(R_{2.1}^{\epsilon,\alpha})_{\overline{b}}$ contained in $\mathcal{T}(D)$.
    
    Secondly, a cell of the form $\big(R^{\epsilon,i}_{2.2}\big)_{(b^1,b^2,b^3,b^4,b^5)}$ should only be included in $\mathcal{T}(D)$ if, again, (\ref{eq: linear residue b2 determined}) holds, and moreover, $\mu_{i_1}(b^1_{\epsilon_1},b^4)$ and $\mu_{i_2}(b^1_{\epsilon_1},b^5)$ are consecutive elements in 
    \begin{align*}
        \bigcup_{i\in\{1,2\}}\{\mu_i(b^1_{\epsilon_1},b): v(b_{i'}-b^1_{\epsilon_1})<v(b^3_{\epsilon_2}-b^1_{\epsilon_1})=v(b_{i}-b^1_{\epsilon_1})\},
    \end{align*}
    where $1'=2$ and $2'=1$.
    Indeed, this is sensible because if $a\in \psi^K(K;b^1_{\epsilon_1},b^2_{\epsilon_2},b^3_{\epsilon_3})$ then $\mu_i(a,b)=\mu_i(b^1_{\epsilon_1},b)$ whenever $b$ is as in the set above. 
    Fixing $b^1$, $b^3$, $b^4$, $\epsilon$ and $i$ this leaves at most one cell of the form  $\big(R_{2.2}^{\epsilon,i}\big)_{(b^1,b^*,b^3,b^4,b')}$ in $\mathcal{T}(D)$.
    Thus, after removing unnecessary cells, $\mathcal{T}(D)$ contains $O(|B|^3)$-many cells of the form $\big(R_{2.2}^{\epsilon,i}\big)_{\overline{b}}$.
    This concludes the proof.
\end{proof}

We are now ready to prove the central ingredient for the point-line incidence bound claimed in Theorem \ref{thm: point line explicit bound}.
Recall that the point-line incidence graph $I\subseteq F^2\times F^2$ is precisely the interpretation $\varphi_I^F$ of the formula $\varphi_I(x_1,x_2;y_1,y_2)\equiv x_2=x_1y_1+y_2$.

\begin{prop}\label{proposition: DCD for point line}
    The formula $\varphi_I(x,y)$ has a piecewise $\mathbf{k}$-still of arity seven.
    Moreover, this induces a DCD of exponent five for $\varphi_I^F$, with a constant factor $C(q)$ polynomial in $q$.
\end{prop}
\begin{proof}
    Let $\psi_1(x_2;z^1,z^2)$, $\psi_2(x_2;z^1,z^2,z^3)$, $\psi_3(x_2;z^1,z^2,u_{\mathbf k})$ and $f_3(x_2,z^1,z^2)$ be as in the statement of Lemma \ref{lemma: step 1 of point line incidence bound}.
    Recall that each $z^j$ is a ternary $\mathbf{K}$-variable, say $z^j=(z^j_1,z^j_2,z^j_3)$.
    Let $z^j_{23}$ be the variable $(z^j_2,z^j_3)$.
    We write $\overline{z}^2$ for the variable $(z^1_{23},z^2_{23},z^3_{23})$ and $\overline{z}^1=\overline{z}^3$ for the variable $(z^1_{23},z^2_{23})$.
    Let $\psi_1'(x_2;x_1,\overline{z}^1)$, $\psi_2'(x_2;x_1,\overline{z}^2)$, and $\psi_3'(x_2;x_1,\overline{z}^3,u_{\mathbf k})$ be obtained by replacing $z^1_1$, $z^2_1$, and $z^3_1$ by $x_1$ in $\psi_1$, $\psi_2$, and $\psi_3$, respectively.
    Define $f'_3(x_2;x_1,\overline{z})$ similarly.
    We write $u_{\mathbf k}^3=u_{\mathbf k}$, and $u_{\mathbf k}^1=u_{\mathbf k}^2=\emptyset$ where $\emptyset$ is the empty variable.

    In the next step, we find formulas $\theta_i(x_1;y,\overline{z}^i,u_{\mathbf k}^i)$, expressing that the set defined by $\psi_i(x_2;x_1,\overline{z}^i,u_{\mathbf k}^i)$ does not contain $\tau(x_1,y):=x_1y_1+y_2$. For $i\leq 3$, define
    \begin{align*}
        &\theta_i(x_1;y,\overline{z}^i,u_{\mathbf k}^i):\equiv\neg\psi_i'\big(\tau(x_1,y),x_1,\overline{z}^i,u_{\mathbf k}^i \big).
    \end{align*}
    For $i\leq 2$, the partitioned formula $\theta_i$ is a linear basic valuative formula, and therefore has a piecewise $\mathbf{k}$-still $(\chi_{i,j}(x_1;s^i;w^i_{\mathbf k}))_{j\leq \ell_i}$ of arity three, by Lemma \ref{lem: linear swiss cheese}.
    On the other hand, the partitioned formula $\theta_3$ is a linear basic residue formula, and therefore has a piecewise $\mathbf{k}$-still $(\chi_{3,j}(x_1;s^3;w^3_{\mathbf k}))_{j\leq \ell_3}$ of arity five, by Lemma \ref{lemma: linear residue images}.
    In particular, $s^1$ and $s^2$ are concatenations of three disjoint copies of $(y,\overline{z}^1)$ and $(y,\overline{z}^2)$, respectively, while $s^3$ is a concatenation of five disjoint copies of $(y,\overline{z}^3,u_{\mathbf k})$.
    Write $\overline{y}^i$ for the variable consisting of only the (three or five) disjoint copies of $y$ contained in $s^i$. 
    For each $j\leq \ell_i$, let $\chi_{i,j}'(x_1;\overline{y}^i,w^i_{\mathbf k};\overline{z}^i,u_{\mathbf k}^i)$ be obtained by replacing in
    $\chi_{i,j}(x_1,s^i,w^i_{\mathbf k})$ each of the disjoint copies of the variable $(\overline{z}^i,u_{\mathbf k}^i)$ coming from $s^i$ with the variable $(\overline{z}^i,u_{\mathbf k}^i)$ itself.

    Finally, we define, for each $i\leq 3$ and $j\leq \ell_i$,
    \[\zeta_{i,j}(x_1,x_2;\overline{z}^i,\overline{y}^i;\,u_{\mathbf k}^i ,w^i_{\mathbf k}):\equiv \psi_i'(x_2,x_1,\overline{z}^i,u_{\mathbf k}^i )\wedge \chi_{i,j}'(x_1,\overline{y}^i,w^i_{\mathbf k},\overline{z}^i,u_{\mathbf k}^i).\]
    To cover the case where $\tp_{\varphi_I}(a/B)\vdash \varphi_I(a,b)$ for some $b\in B$, we set $\ell_4:=1$, $\overline{z}^4=z^1_{23}$, $u_{\mathbf k}^4=w_{\mathbf k}^4=\overline{y}^4=\emptyset$, and
    $\zeta_{4,1}(x_1,x_2;\overline{z}^4,\overline{y}^4;u_{\mathbf k}^4,w_{\mathbf k}^4):\equiv x_2=x_1z^1_2+z^1_3$.\newline

    \underline{Claim:} The family $\big(\zeta_{i,j}(x_1,x_2;\overline{z}^i,\overline{y}^i;u_{\mathbf k}^i,w^i_{\mathbf k})\big)_{i\leq 4,j\leq \ell_i}$ is a piecewise $\mathbf{k}$-still for $\varphi_I$ of arity seven.
    \begin{proof}
        Let $(K,v)\models \ACVF$, $a\in K^2$, and $B\fin K^2$.
        If $\varphi_I(a,b)$ holds for some $b\in B$, then $\varphi_I(x;b)\equiv \zeta_{4,1}(x;b)$ clearly isolates $\tp_{\varphi_I}(a/B)$.
        Suppose $\varphi_I(a,b)$ does not hold for any $b\in B$.
        By our choice of $(\psi'_i)_i$ and $f'_3$, and by Lemma \ref{lem: relative USHDs with parameters}, there is some $i\leq 3$ and a tuple $c$ from $B$ such that for $\alpha:=f_3'(a_2,a_1,c)$ (or $\alpha=\emptyset$ if $i\leq 2$),
        \begin{align}\label{eq: isolating property}
            (K,v)\models \forall x_2\big(\psi_i'(x_2;a_1,{c},\alpha) \rightarrow \bigwedge_{b\in B}  x_2\neq \tau(a_1,b)\big).
        \end{align}
        Equivalently,
        \[ (K,v)\models \bigwedge_{b\in B} \underbrace{\neg\psi_i'(\tau(a_1,b),a_1,c,\alpha)}_{\equiv \theta_i(a_1,b,c,\alpha) }.\]
        
        By choice of $(\chi_{i,j})_{i,j}$, we can find suitable parameters $d$ from $B$ and $\beta$ from $\mathbf k(K)$ such that $\chi_{i,j}'(x_1;d,\beta,c,\alpha)$ isolates $\tp_{\theta_i}(a_1/\{a_1\}\times B\times\{(c,\alpha)\})$.
        The statement follows.
        To see that the arity is as stated, note that $(\overline{z}^1,\overline{y}^1)$ contains five copies of $y$, $(\overline{z}^2,\overline{y}^2)$ contains six, $(\overline{z}^3,\overline{y}^3)$ contains seven, and $(\overline{z}^4,\overline{y}^4)$ contains one.
    \end{proof}
  
    For the ``Moreover"-part of the statement, recall that $(F,v)$ is a valued field with finite residue field $\mathbb F_q$.
    For each $i\leq 4$, $j\leq\ell_i$, and $(\alpha,\beta)\in \mathbb F_q^{|u_{\mathbf k}^i|+|w_{\mathbf k}^i|}$, define
    \begin{align*}
        R_{i,j}^{\alpha,\beta}:=\Big\{(a,c,d)\in F^2\times \big((F^2)^{|\overline{z}^i|/2}\times (F^2)^{|\overline{y}^i|/2}\big):\ (K,v)\models \zeta_{i,j}(a;c,d;\alpha,\beta)\Big\}.       
    \end{align*}
    
    By the claim, these form a DCD $\mathcal{T}$ for $I=(\varphi_I)^F$.
    Finally, we show that the exponent of this DCD can be reduced from the trivial bound seven (coming from the arity of the piecewise $\mathbf{k}$-still) to five.
    For $i=1$ and $i=4$, just counting parameters shows that there are at most $O(|B|^5)$-many cells of the form $(R_{i,j}^{(\alpha,\beta)})_{(c,d)}$.
    For $i= 2$, the ``Moreover"-part of Lemma \ref{lem: linear swiss cheese} shows that the family $(\chi_{2,j}')_{j\leq \ell_2}$ induces a DCD $\mathcal{T}_2$ of exponent \emph{two} with constant factor polynomial in $q$, for the partitioned formula $\theta_2$. 
    We include a cell of the form $(R_{2,j}^{(\alpha,\beta)})_{(c,d)}$ in $\mathcal{T}(B)$ only if the cell $(\chi_{2,j}^\beta)_d$ is included in $\mathcal{T}_2(B\times\{(c,\alpha)\})$. 
    With this, there are at most $O_q(|B|^{|\overline{z}^2|/2+2})=O_q(|B|^5)$-many cells of this form contained in $\mathcal{T}(B)$.
    Finally, by Lemma \ref{lemma: linear residue images}, the family $(\chi_{3,j}')_{j\leq \ell_3}$ induces a DCD $\mathcal{T}_3$ of exponent \emph{three} with a constant factor polynomial in $q$, for the partitioned formula $\theta_3$, and we include a cell of the form $(R_{3,j}^{(\alpha,\beta)})_{(c,d)}$ in $\mathcal{T}(B)$ only if the cell $(\chi_{3,j}^\beta)_d$ is included in $\mathcal{T}_3(B\times\{(c,\alpha)\})$.
    This leaves at most $O_q(|B|^{|\overline{z}^3|/2+3})=O_q(|B|^5)$-many cells of this form in $\mathcal{T}(B)$.
    In both cases, the constant factor is polynomial in $q$ and we conclude that $\mathcal{T}(B)$ has exponent five, witnessed by a constant factor which is polynomial in $q$.
\end{proof}

With the distal Zarankiewicz bound of \cite{chernikov2023modeltheoreticelekesszabostableominimal} (Theorem \ref{thm: distal zarankiewicz}), we can  immediately derive our point-line incidence bound for the field $(F,v)$ with residue field $\mathbb F_q$.

\begin{proof}[Proof of Theorem \ref{thm: point line explicit bound}]
    Let $I$ be point-line incidence in $F$, and let $A\fin F^2$ and $B\fin F^2$.
    Since $I$ is $K_{2,2}$-free, an application of Theorem \ref{thm: distal zarankiewicz} with the DCD coming from Proposition \ref{proposition: DCD for point line} gives the bound
    \[O_{q}(|A|^{\frac{8}{9}}|B|^{\frac{5}{9}}+|A|+|B|),\]
    since we have $s=d=2$ and $t=5$ in (\ref{eq: distal zarankiewicz bound}).
\end{proof}

\begin{remark}
    Using Remark \ref{remark: polynomial constant}, the constant factor in the bound given by Theorem \ref{thm: point line explicit bound} is polynomial in $q$.
\end{remark}


\bibliographystyle{amsplain}
\bibliography{references}

\end{document}